\documentclass[a4paper,10pt,reqno, english]{amsart}  
\usepackage{amssymb}
\usepackage{graphicx}
\usepackage{amsmath,amsthm}
\usepackage{amsfonts,amssymb,enumerate}
\usepackage{url,paralist}
\usepackage{mathtools}
\usepackage{cancel}
\usepackage[arrow,curve,matrix,tips,2cell]{xy}  
  \SelectTips{eu}{10} \UseTips
  \UseAllTwocells

\usepackage{tikz}
\usepackage{hyperref}
\usepackage{color}
\usepackage{enumerate,amssymb}

\theoremstyle{plain}
\newtheorem{theorem}{Theorem}[section]

\newtheorem{lemma}[theorem]{Lemma}
\newtheorem{claim}[theorem]{Claim}
\newtheorem{corollary}[theorem]{Corollary}
\newtheorem{proposition}[theorem]{Proposition}

\newtheorem*{theorem*}{Theorem}

\theoremstyle{definition}
\newtheorem{definition}[theorem]{Definition}

\newtheorem{example}[theorem]{Example}

\newcommand{\BIGOP}[1]{\mathop{\mathchoice%
{\raise-0.22em\hbox{\huge $#1$}}%
{\raise-0.05em\hbox{\Large $#1$}}{\hbox{\large $#1$}}{#1}}}
\newcommand{\bigtimes}{\BIGOP{\times}}

\newcommand\RP{\mathbb{R}{\rm P}}
\newcommand{\RR}{\mathbb{R}}

\newcommand{\CC}{\mathbb{C}}
\newcommand{\NN}{\mathbb{N}}
\newcommand{\ZZ}{\mathbb{Z}}
\newcommand{\FF}{\mathbb{F}}
\newcommand{\BB}{\mathrm{B}}
\newcommand{\EE}{\mathrm{E}}

\newcommand\Sym{\mathfrak S}
\newcommand\C{\mathfrak C}
\newcommand{\dimaff}{\dim_{\affine}}

\newcommand{\cat}{\operatorname{cat}}
\newcommand{\secat}{\operatorname{secat}}
\newcommand{\hofib}{\operatorname{hofib}}
\newcommand{\id}{\operatorname{id}}

\newcommand{\colim}{\operatorname{colim}}

\newcommand{\affine}{\operatorname{aff}}
\newcommand{\h}{\operatorname{\mathfrak h}}

\begin{document}
% -------------------------------------------------------------------%

% -------------------------------------------------------------------%
\title[On complex highly regular embeddings]{On complex highly regular embeddings and the extended Vassiliev conjecture}

% -------------------------------------------------------------------%

% -------------------------------------------------------------------%
\author[Blagojevi\'c]{Pavle V. M. Blagojevi\'{c}} 
\thanks{The research by Pavle V. M. Blagojevi\'{c} leading to these results has
        received funding from the Leibniz Award of Wolfgang L\"uck granted by DFG.
        Also supported by the grant ON 174008 of the Serbian Ministry of Education and Science.}
\address{Mathemati\v cki Institut SANU, Knez Mihailova 36, 11001 Beograd, Serbia}
\email{pavleb@mi.sanu.ac.rs} 
\curraddr{Institut f\"ur Mathematik, FU Berlin, Arnimallee 2, 14195 Berlin, Germany}
\email{blagojevic@math.fu-berlin.de}
\author[Cohen]{Frederick R. Cohen}
\thanks{The research by Frederick R. Cohen leading to these results has received funding from University of Rochester and 
Institute for Mathematics and its Applications (IMA) of University of Minnesota.}
\address{Department of Mathematics, University of Rochester, Rochester, NY 14625, U.S.A.}
\email{cohf@math.rochester.edu} 
\author[L\"uck]{Wolfgang L\"uck} 
\thanks{The research by Wolfgang L\"uck leading to these results has received funding from a Leibniz
        Award granted by  DFG}
\address{Mathematisches Institut der Universit\"at Bonn, Endenicher Allee 60, 53115 Bonn, Germany} 
\email{wolfgang.lueck@him.uni-bonn.de} 
\author[Ziegler]{G\"unter M. Ziegler} 
\thanks{The research by G\"unter M. Ziegler leading to these results has received funding from the 
        European Research Council under the European Union's Seventh Framework Programme (FP7/2007-2013) / ERC   Grant agreement no.~247029-SDModels.}  
\address{Institut f\"ur Mathematik, FU Berlin, Arnimallee 2, 14195 Berlin, Germany} 
\email{ziegler@math.fu-berlin.de} 
\date{March 12, 2015; revised October 27, 2015} 

% -------------------------------------------------------------------%

% -------------------------------------------------------------------%

\dedicatory{Dedicated to the memory of Samuel Gitler}
% -------------------------------------------------------------------%

\maketitle

\begin{abstract}
A continuous map $\CC^d\longrightarrow\CC^N$ is a complex $k$-regular embedding if any $k$ pairwise distinct points in $\CC^d$ are mapped by $f$ into $k$ complex linearly independent vectors in $\CC^N$.
The existence of such maps is closely connected with classical problems of algebraic/differential topology, such as embedding/immersion problems.
Our central result on complex $k$-regular embeddings extends results of Cohen \& Handel (1978), Chisholm (1979) and  Blagojevi\'c, L\"uck \& Ziegler (2013) on real $k$-regular embeddings. 
We give the following lower bounds for the existence of complex $k$-regular embeddings
{\em 
\begin{compactitem}[~~]
\item Let $p$ be an odd prime, $k\geq1$ and $d=p^t$ for $t\geq1$. 
         If there exists a complex $k$-regular embedding $\CC^d\longrightarrow\CC^N$, then $ d(k-\alpha_p(k))+\alpha_p(k)\leq N$. Here $\alpha_p(k)$ denotes the sum of coefficients in the $p$-adic expansion of $k$.
\end{compactitem}}
\noindent
These lower bounds are obtained by modifying the framework of Cohen \& Handel (1978) and a study of Chern classes of complex regular representations. 
As a main technical result we establish for this an extended Vassiliev conjecture, the following upper bound for the height of the cohomology of an unordered configuration space:
{\em
\begin{compactitem}[~~]
\item If $d\geq 2$ and $k\geq 2$ are integers, and $p$ is an odd prime. Then 
\[
 \h(H^*(F(\RR^d,k)/\Sym_k;\FF_p))\leq\min\{ p^t: 2p^t\geq d \}.
\]
\end{compactitem}}

\noindent
Furthermore,  we give similar lower bounds for the existence of complex $\ell$-skew embeddings $\CC^d\longrightarrow\CC^N$, for which we require that the images of the tangent spaces at any $\ell$ distinct points are skew complex affine subspaces of $\CC^N$.

\noindent
In addition we give improved lower bounds for the Lusternik--Schnirelmann category of $F(\CC^d,k)/\Sym_k$ as well as for the sectional category of the covering $F(\CC^d,k)\longrightarrow F(\CC^d,k)/\Sym_k$.
\end{abstract}

% -------------------------------------------------------------------%

%%%%%%%%%%%%%%%%%%%%%%%%%%%%%%%%%%%%%%%%%%%%%%%%%%%%%%%%%%%%%%%%%%%%%%%%%%%%%%%%%%%%%
%%%%%%%%%%%%%%%%%%%%%%%%%%%%%%%%%%%%%%%%%%%%%%%%%%%%%%%%%%%%%%%%%%%%%%%%%%%%%%%%%%%%%
\section{Introduction}
%%%%%%%%%%%%%%%%%%%%%%%%%%%%%%%%%%%%%%%%%%%%%%%%%%%%%%%%%%%%%%%%%%%%%%%%%%%%%%%%%%%%%
%%%%%%%%%%%%%%%%%%%%%%%%%%%%%%%%%%%%%%%%%%%%%%%%%%%%%%%%%%%%%%%%%%%%%%%%%%%%%%%%%%%%%

Let $V$ be a finite dimensional vector space over the field of real or of complex numbers.

For $X$ a topological space and $k\ge1$, 
a continuous map  $f\colon X\longrightarrow V$ is a {\em  real (complex) $k$-regular embedding} if any $k$ pairwise distinct points in $X$  are mapped to 
$k$ vectors that are linearly independent in the real (complex) vector space $V$.
 
For $M$ a smooth complex $d$-dimensional manifold and $\ell\ge1$,
a smooth complex embedding $f \colon M\longrightarrow\CC^N$ is a \emph{complex $\ell$-skew embedding} if for every collection $y_1,\ldots,y_{\ell}$ of pairwise distinct points in $M$ the complex affine subspaces
   \[
   (\iota\circ df_{y_1})(T_{y_1}M),\ldots ,(\iota\circ df_{y_{\ell}})(T_{y_{\ell}}M)
   \]
   of $\CC^N$ are affinely independent. Here $df\colon TM\longrightarrow T\CC^N$ is the complex differential map induced by $f$, and 
$ \iota \colon T\CC^N \longrightarrow \CC^N$
maps a tangent vector $v \in T_x\CC^N$ at $x \in \CC^N$
 to  $x +v$ where we use the standard identification $T_x\CC^N = \CC^N$.

\smallskip
The study of real $k$-regular embeddings was initiated by Borsuk in 1957 \cite{borsuk}. 
The first non-trivial lower bound for the existence of real $2k$-regular embeddings between real vector spaces, obtained by dimension count, was given by Boltjanski{\u\i}, Ry{\v{s}}kov \& {\v{S}}a{\v{s}}kin in 1963~\cite{boltjanskii-ryskov-saskin}:
\begin{compactitem}[~~~]
\item{\em  If a real $2k$-regular  embedding $\RR^d\longrightarrow\RR^N$ exists, then $(d+1)k\leq N$.}
\end{compactitem}
In 1978  F. Cohen and Handel \cite{cohen-handel} related the problem of the existence of a real $k$-regular embedding $X\longrightarrow\RR^N$ with the existence of an $\Sym_k$-equivariant map between the configuration space $F(X,k)$ and the Stiefel manifold $V_k(\RR^N)$ of all $k$-frames in $\RR^N$.
The key equivalence between the existence of $\Sym_k$-equivariant maps $F(X,k)\longrightarrow V_k(\RR^N)$ and the existence of a particular inverse of the regular representation vector bundle over $F(X,k)/\Sym_k$ allowed them to employ methods of algebraic topology.
They proved that:
\begin{compactitem}[~~~]
\item{\em  If a real $k$-regular  embedding $\RR^2\longrightarrow\RR^N$ exists, then $2k-\alpha(k)\leq N$, where $\alpha(k)$ denotes the number of ones in the dyadic expansion of $k$.}
\end{compactitem}
In 1979 Chisholm \cite[Thm.\,2]{chisholm} extended the result of Cohen \& Handel as follows:
\begin{compactitem}[~~~]
\item {\em Let $d$ be a power of $2$ and $k\geq 1$. 
                  If a real $k$-regular  embedding $\RR^d\longrightarrow\RR^N$ exists, then $d(k-\alpha (k))+\alpha (k)\leq N$.}
\end{compactitem}
In 1980's and 1990's the importance of the existence of real $k$-regular embeddings for approximation theory motivated further studies of 
Handel~\cite{handel,handel-segal,handel96}  and Vassiliev~\cite{vassiliev92book,vassiliev92,vassiliev96}.

\noindent
Only recently, in \cite[Thm.\,2.1]{blagojevic-luck-ziegler-2}, the restriction on the dimension in Chisholm's 1979 result has been removed:
\begin{compactitem}[~~~]
\item {\em Let $d\geq2$ and $k\geq1$. If a real $k$-regular  embedding $\RR^d\longrightarrow\RR^N$ exists, then $d(k-\alpha (k))+\alpha (k)\leq N$.}
\end{compactitem}
The methods employed for this extension of Chisholm's result also led to a considerable improvement of the lower bounds for the existence of real $\ell$-skew and $k$-regular-$\ell$-skew embeddings, originally studied by Ghomi \& Tabachnikov \cite{ghomi-tabachnikov} (case $\ell=2$) and Stojanovi\'c \cite{stojanovic}.

\smallskip
In this paper, following the setup of Cohen \& Handel from 1987, we study the existence of complex $k$-regular and $\ell$-skew embeddings. 
Our first two new results, Theorems~\ref{theorem_0} and \ref{theorem:Main-02}, are consequences of our  results on real $k$-regular and $\ell$-skew embeddings \cite{blagojevic-luck-ziegler-2}:
\begin{compactitem}[~~~]
\item
\begin{compactenum}[\rm a.]
\item Let $d\geq 1$ and $k\geq 1$ be integers. 
If a complex $k$-regular embedding $\RR^d\longrightarrow\CC^N$ exists, then $N\geq \tfrac{1}{2}(d(k-\alpha(k))+\alpha(k))$.
\item Let $d\geq 1$ and $\ell\geq 1$ be integers. 
If a complex  $\ell$-skew embedding $\CC^d\longrightarrow\CC^{N}$ exists, then 
  $N\geq \tfrac{1}{2}\big( (2^{\gamma(d)+1}-2d-1)(\ell-\alpha(\ell))+2(d+1)\alpha(\ell)-\ell\big)-1$,
  where  $\gamma(d)=\lfloor\log_2d\rfloor+1$.
\end{compactenum}
\end{compactitem}
Using a result of Cohen from 1976~\cite[Vanishing theorem 8.2]{cohen}~\cite[Thm.\,6.1]{blagojevic-luck-ziegler-1}, we prove in Theorems~\ref{theorem_1} and \ref{theorem:Main-2}:
\begin{compactitem}[~~~]
\item 
\begin{compactenum}[\rm A.]
\item {\em Let $d\geq 1$ be an integer, and let $p$ be an odd prime. If a complex $p$-regular embedding $\RR^d\longrightarrow\CC^N$ exists, then $\lfloor \tfrac {d+1}{2}\rfloor (p-1)+1\leq N$.
\item  Let $d\geq1$ be an integer and let $\ell\geq 3$ be a prime.  
If a complex  $\ell$-skew embedding $\CC^d\longrightarrow\CC^{N}$ exists, then  $(\ell-1)(d+f(d,\ell)+1)+d\leq N$, where 
  $f(d,\ell):=\max\{k:0\leq k\leq d-1\textrm{ and }\ell \nmid{d+k\choose d} \}$.}
\end{compactenum}
\end{compactitem}
In \cite{buczynski et al} the existence of 
real and complex $k$-regular maps $\RR^d \longrightarrow \mathbb{R}^N$ and $\mathbb{C}^d \longrightarrow \mathbb{C}^N$ is proven for $N = dk-1$ and any 
$k$, or for $N=d(k-1)+1$ and $k\leq 9$.

A key technical result of this paper -- the extended Vassiliev conjecture, 
Theorem \ref{theorem:hight_bound_for_prime_2} for $p=2$ and Theorem \ref{theorem:hight_bound_for_odd} for odd primes -- gives estimates for the height of the cohomology algebra of the unordered configuration space $F(\RR^d,k)/\Sym_k$ with the coefficients in the field $\FF_p$:
\begin{compactitem}[~~~]
\item 
\begin{compactenum}[\rm 1.]
\item Let $d\geq 2$ and $k\geq 2$ be integers.
Then
\[
\h(H^{*}(F(\RR^d,k)/\Sym_k;\FF_2))\leq \min\{ 2^t: 2^t\geq d \}.
\]

\item {\em Let $d\geq 2$ and $k\geq 2$ be integers, and let $p$ be an odd prime. Then 
\[
 \h(H^*(F(\RR^d,k)/\Sym_k;\FF_p))\leq\min\{ p^t: 2p^t\geq d \}.
\]
 
}
\end{compactenum}
\end{compactitem}
In particular, for $p=2$ and $d$ power of two part 1. of these results settle a conjecture by
Vassiliev \cite[Conj.\,2,\,pp.\,210]{vassiliev92}~\cite[Conj.\,2,\,pp.\,75]{vassiliev92book}.

Using these estimates as a black box, in Theorems~\ref{theorem_2}  and \ref{theorem:Main-3} we prove the following Chisholm-like lower bounds for the existence of $k$-regular and $\ell$-skew embeddings over $\CC$:
\begin{compactitem}[~~~]
\item 
\begin{compactenum}[\rm i.]
{\em 
\item Let $p$ be an odd prime, $k\geq1$ and $d=p^t$ for $t\geq1$. 
         If there exists a complex $k$-regular embedding $\CC^d\longrightarrow\CC^N$, then $ d(k-\alpha_p(k))+\alpha_p(k)\leq N$. Here $\alpha_p(k)$ denotes the sum of coefficients in the $p$-adic expansion of $k$.

\item Let $p$ be an odd prime, $\ell\geq1$ and $d=p^t$ for $t\geq1$.   
          If there exists a complex  $\ell$-skew embedding $\CC^d\longrightarrow\CC^{N}$, then $(d-1)(\ell-\alpha_p(\ell))+(d+1)\ell-1\leq N$.
}
\end{compactenum}
\end{compactitem}

In addition, in Corollaries~\ref{cor:LS-category} and \ref{cor:secat-category} we present improved lower bounds for the Lusternik--Schnirelmann category of  $F(\CC^d,k)/\Sym_k$ as well as for the sectional category of the covering $F(\CC^d,k)\longrightarrow F(\CC^d,k)/\Sym_k$:
\begin{compactitem}[~~~]
\item 
\begin{compactenum}[\rm I.]
{\em 
\item $\cat (F(\CC^d,k)/\Sym_k)\geq \max\{2(d-1)(k-\alpha_p(k)): p\text{ is a prime}\}$. 
\item  Let $d\geq 1$ be an integer and $k$ be a power of an odd prime $p$. Then
\[
 (2d-2)(k-1)\leq\secat\big(F(\CC^d,k)\longrightarrow F(\CC^d,k)/\Sym_k\big)\leq (2d-1)(k-1).
\]
}
\end{compactenum}
\end{compactitem}
These questions were extensively studied from different aspects by Vassiliev~\cite{vassiliev96}, De Concini, Procesi \& Salvetti \cite{deConcini-procesi-salvetti}, Arone~\cite{arone2006}, and Roth~\cite{Roth}.

Using the methods of this paper one can easily derive lower bounds for the existence of complex $k$-regular-$\ell$-skew embedding as well as complex analogues of results by Vassiliev \cite{vassiliev92,vassiliev96}.

Finally, let us note that the existence of $k$-regular embeddings and $\ell$-skew maps (in the real and in the complex versions) is closely related to a number of hard, classical algebraic/differential topology problems, 
such as the generalized vector field problem, immersion problems for projective spaces, and the existence of non-singular bilinear forms,
see Ghomi \& Tabachnikov \cite{ghomi-tabachnikov}.

\medskip
\noindent
{\bf Acknowledgments.} The Institute for Mathematics and its Applications (IMA) of University of Minnesota provided a perfect environment for work on  this paper.
We are grateful to Anthony Bahri, Martin Bendersky, Mark Grant, Peter Landweber, John McCleary, and  Yuliya Semikina for their useful comments,
hospitality and support. 
We also thank the anonymous referee for his careful work.

%%%%%%%%%%%%%%%%%%%%%%%%%%%%%%%%%%%%%%%%%%%%%%%%%%%%%%%%%%%%%%%%%%%%%%%%%%%%%%%%%%%%%
%%%%%%%%%%%%%%%%%%%%%%%%%%%%%%%%%%%%%%%%%%%%%%%%%%%%%%%%%%%%%%%%%%%%%%%%%%%%%%%%%%%%%
\section{Some ingredients}
%%%%%%%%%%%%%%%%%%%%%%%%%%%%%%%%%%%%%%%%%%%%%%%%%%%%%%%%%%%%%%%%%%%%%%%%%%%%%%%%%%%%%
%%%%%%%%%%%%%%%%%%%%%%%%%%%%%%%%%%%%%%%%%%%%%%%%%%%%%%%%%%%%%%%%%%%%%%%%%%%%%%%%%%%%%

In this section we collect facts that will be used in the rest of the paper, in particular in the proofs of Theorems~\ref{theorem:hight_bound_for_prime_2} and~\ref{theorem:hight_bound_for_odd}.
For this we
\begin{compactitem}
\item recall the relationship between $\colim_{k\geq 0}F(\RR^d,k)/\Sym$ and $\Omega^d_0 S^d$, Section~\ref{sec:Unordered configuration space and iterated loop spaces};
\item summarize relevant properties of the Araki--Kudo--Dyer--Lashof homology operations, Section~\ref{sec:Araki--Kudo--Dyer--Lashof operations};
\item introduce the space $Q(X):=\colim_{d\geq 0}\Omega^d\Sigma^d(X)$ and describe its homology in the language of Araki--Kudo--Dyer--Lashof operations, Section~\ref{sec:The space $Q(X)$};
\item study the height of $H_*(Q(F(\RR^d,p)/\Sym_p);\FF_p)$ when $p$ is an odd prime, Section~\ref{sec:On the homology and cohomology of unordered configuration space}; and finally
\item prove a comparison lemma for the heights in an exact sequence of Hopf algebras, Section~\ref{sec:Height comparison lemma for Hopf algebras}.
\end{compactitem}

\medskip
Let $A$ be an algebra over a field $\FF$.
The {\em height of an element} $a\in A{\setminus}A^*$ is a natural number or infinity:
\[
\h(a):=\min \{n\in\NN : a^n=0\},
\]
where $A^*$ denotes the subgroup of all invertible elements of the algebra $A$.
The {\em height of the algebra} $A$ is defined to be
\[
\h(A):=\max \{\h(a):a\in A{\setminus}A^*\}.
\]

% -------------------------------------------------------------------%
% -------------------------------------------------------------------%
\subsection{Unordered configuration space and iterated loop spaces}
\label{sec:Unordered configuration space and iterated loop spaces}
% -------------------------------------------------------------------%
% -------------------------------------------------------------------%

The first ingredient we need is the following result by Cohen, May and 
Segal~\cite[Cor.\,3.3, taking $X=S^0$]{cohen}~\cite{may}~\cite[Thm.\,3]{segal}.
Consider the iterated loop space
\[
 \Omega^d S^d:=\{f\colon S^d\longrightarrow S^d : f\text{ a continuous map}, f(*)=*\},
\]
where $*\in S^d$ denotes the base point.
It decomposes into connected components with respect to the degree of a map $f\in  \Omega^d S^d$, 
i.e., $\Omega^d S^d=\coprod_{q\in\ZZ}\Omega^d_{q} S^d$, where
\[
\Omega^d_{q} S^d:=\{f\in\Omega^d S^d : \deg(f)=q\}.
\]

\begin{theorem}\label{theorem:tools-1}
For every integer $d\geq 1$ there is a continuous map
\[
\alpha_{d}\colon\colim_{k\geq 0}F(\RR^d,k)/\Sym_k\longrightarrow\Omega^d_0 S^d 
\]
that induces an isomorphism in homology
\[
(\alpha_{d})_*\colon H_*(\colim_{k\geq 0}F(\RR^d,k)/\Sym_k;R)
\longrightarrow
H_*(\Omega^d_0 S^d;R),
\]
with any simple coefficients $R$.\newline
Furthermore, for every $\ell\geq 0$ there is a commutative diagram:
\[
\xymatrix@!C
{
\colim_{k\geq0}F(\RR^d,k)/\Sym_k\ar[r]^-{\alpha_{d}}\ar@{^{(}->}[d]& \Omega^d_0 S^d\ar[d]\\
\colim_{k\geq0}F(\RR^{d+\ell},k)/\Sym_k\ar[r]^-{\alpha_{d+\ell}}\ar@{^{(}->}[d]& \Omega^{d+\ell}_0 S^{d+\ell}\ar[d]\\
\colim_{k\geq0}F(\RR^\infty,k)/\Sym_k\ar[r]^-{\alpha_{\infty}}& \Omega^\infty_0 S^\infty,
}
\]
where $\Omega^\infty_0 S^\infty:=\colim_{d\geq 0}\Omega^d_0 S^d$.

\noindent
Finally, for a prime $p$, the composition of rightmost vertical maps
\[
\Omega^d_0 S^d\longrightarrow\Omega^\infty_0 S^\infty
\]
induces a monomorphism in homology with coefficients in $\FF_p$ if and only if it induces an
epimorphism in cohomology with coefficients in $\FF_p$ if and only if
\begin{compactenum}[\rm (1)]
\item $p=2$, or
\item $p>2$ and $d$ is odd.
\end{compactenum}
\end{theorem}

A direct consequence of the previous theorem is the following corollary.
\begin{corollary}\label{corollary:tools-1}
Let $d\geq 3$ be an odd integer, $k\geq 2$ be an integer, and $\FF$ be any field.
The composition 
\[
\gamma_{d,k}:=\alpha_{d}\circ i_{d,k}\colon F(\RR^d,k)/\Sym_k\longrightarrow \Omega^d_{0}S^d,
\]
of the structural map 
$i_{d,k}\colon F(\RR^d,k)/\Sym_k\longrightarrow\colim_{k\geq0}F(\RR^d,k)/\Sym_k$ 
and the map $\alpha_{d}$ induces an epimorphism of cohomology  algebras
\[
\gamma_{d,k}^*\colon H^*(\Omega^d_{0}S^d;\FF)\longrightarrow H^*(F(\RR^d,k)/\Sym_k;\FF).
\]
\end{corollary}
\begin{proof}
The structural map $i_{d,k}\colon F(\RR^d,k)/\Sym_k\longrightarrow\colim_{k\geq0}F(\RR^d,k)/\Sym_k$ has a stable section, namely $\colim_{k\geq0}F(\RR^d,k)/\Sym_k$ stably splits where one of the factors is $F(\RR^d,k)/\Sym_k$, \cite[pp.~28]{cohen: 1995}.
Thus, $i_{d,k}$ induces a monomorphism in homology.
Since by the previous theorem $\alpha_{d}$ induces an isomorphism in homology it follows that ${\gamma_{d,k}}_*$ is a monomorphism in homology, or equivalently $\gamma_{d,k}^*$ is an epimorphism in cohomology.
\end{proof}

% -------------------------------------------------------------------%
% -------------------------------------------------------------------%
\subsection{Araki--Kudo--Dyer--Lashof operations}
\label{sec:Araki--Kudo--Dyer--Lashof operations}
% -------------------------------------------------------------------%
% -------------------------------------------------------------------%

 In this section we review some key properties of the Araki--Kudo--Dyer--Lashof homology operations.
In the case $p=2$ the homology operations were first introduced by Araki \& Kudo \cite{araki-kudo}, while in the case of odd primes it was done by  Dyer \& Lashof \cite{dyer-lahof}.
Peter May in his 1976 paper \cite{may533} expands the work of Araki \& Kudo  and Dyer \& Lashof by transforming homology operations into the more elegant and efficient so called upper notation.
Our review follows the presentation of May in \cite{may533}.

The spaces we consider are compactly generated weak Hausdorff and have a non-degenerate base-point.
Let $\C$ be an $E_{\infty}$ operad, let $X$ be a $\C$-space (or simply $E_{\infty}$-space), and let $p$ be a prime.
Then there exists a sequence of homomorphisms
\begin{eqnarray*}
Q^s\colon H_q(X;\FF_2)\longrightarrow H_{q+s}(X;\FF_2),  & \text{if }p=2\\
 Q^s\colon H_q(X;\FF_p)\longrightarrow H_{q+2s(p-1)}(X;\FF_p),& \text{if }p>2
\end{eqnarray*}
where $s\geq 0$.
These maps are called {\em Araki--Kudo--Dyer--Lashof  operations}; they satisfy following properties.
\begin{proposition}
\label{prop:DL-operations}
Let $X$ and $Y$ be $E_{\infty}$-spaces, and let $p$ be a prime.
Then the following hold:
\begin{compactenum}[\rm (1)]
\item The operations $Q^s$ are natural with respect to maps of $\C$-spaces.
\item Let  $x\in H_q(X;\FF_p)$, then
\begin{eqnarray*}
Q^s(x)=0,  & \text{if }s<q\text{ and } p=2,\\
Q^s(x)=0,  & ~\text{if }2s<q\text{ and } p>2.
\end{eqnarray*}
\item Let  $x\in H_q(X;\FF_p)$, then
\begin{eqnarray*}
Q^s(x)=x^2,  & \text{if }s=q\text{ and } p=2,\\
Q^s(x)=x^p,  & ~\text{if }2s=q\text{ and } p>2.
\end{eqnarray*}
\item Let $1\in H_0(X;\FF_p)$ be the class of the base point. Then $Q^0(1)=1$ and $Q^s(1)=0$ for $s>0$.
\item Let $\Delta\colon X\longrightarrow X\times X$ be the diagonal embedding and $\Delta_*$ the corresponding homomorphism in homology.
         The following Cartan formulas hold:
\begin{eqnarray*}
Q^s(x\otimes y)=\sum_{i=0}^sQ^i(x)\otimes Q^{s-i}(y),   & x\otimes y\in H_*(X\times Y;\FF_p),\\
Q^s(xy)=\sum_{i=0}^sQ^i(x)Q^{s-i}(y),   & x, y\in H_*(X;\FF_p),\\
\Delta_*(Q^s(x))=\sum_{i=0}^s\Big(\sum Q^i(x')\otimes Q^{s-i}(x'')\Big),   & x\in H_*(X;\FF_p),\ \Delta_*(x)=\sum x'\otimes x''.
\end{eqnarray*}         
\end{compactenum}
\end{proposition}

\medskip
The following property \cite[Cor.\,2]{dyer-lahof} is of a particular use for us.

\begin{lemma}
\label{lemma:Q(primitve)}
Let $X$ be a path connected $E_{\infty}$-space, and let $p$ be a prime.
If $x\in H_r(X;\FF_p)$ is a primitive homology class, then $Q^s(x)$ is primitive for every $s$.
\end{lemma}
\begin{proof}
The homology class $x\in H_r(X;\FF_p)$ is primitive.
That is,
\[
\Delta_*(x)=x\otimes 1 + 1\otimes x,
\]
where $\Delta\colon X\longrightarrow X\times X$ is the diagonal embedding and $\Delta_*$ the corresponding homomorphism in homology.
From Proposition \ref{prop:DL-operations}\,(4)-(5) we have that
\[
\Delta_*(Q^s(x))=\sum_{i=0}^s\Big(Q^i(x)\otimes Q^{s-i}(1)+Q^i(1)\otimes Q^{s-i}(x)\Big)= Q^s(x)\otimes 1+1\otimes Q^s(x).
\]
Thus $Q^s(x)$ is primitive.
\end{proof}

Next we introduce  a notion of admissible sequences of Araki--Kudo--Dyer--Lashof  operations,  as presented in \cite[Sec.\,2]{may533}.

Let $p=2$, and  let $I = (s_1, \ldots, s_k)$  be a sequence of non-negative integers.
Degree, length, and excess of the sequence $I$ are defined by the formulas
\[
d(I)= \sum_{k\geq 1}s_j,\quad \ell(I) = k,\quad e(I) = s_1-\sum_{2 \leq j\leq k}s_j.
\]
The sequence $I= (s_1, \ldots, s_k)$ determines the homology operation $Q^I = Q^{s_1} \cdots Q^{s_k}$.
The operation $Q^I$, or the sequence $I$,  is called {\em admissible} if $2s_j \geq s_{j-1}$ for $ 2 \leq j \leq k$.

Let $p>2$ be an odd prime, and let $I = (\varepsilon_1,s_1, \ldots, \varepsilon_k, s_k)$ be a sequence on non-negative integers where $\varepsilon_{r}\in\{0,1\}$ and $s_r\geq \varepsilon_r$, for every $1\leq r\leq k$. 
Degree, length, and excess of the sequence $I$ in this case are defined by the formulas
\[
\quad d(I)= \sum_{k\geq 1}(2s_j(p-1) - \varepsilon_j),\quad \ell(I) = k,\quad e(I) = s_1-\sum_{2 \leq j\leq k}(s_j(p-1) - \varepsilon_j).
\]
The sequence $I = (\varepsilon_1,s_1, \ldots, \varepsilon_k, s_k)$ determines the homology operation $Q^I = \partial_p^{\varepsilon_1}Q^{s_1} \cdots \partial_p^{\varepsilon_k}Q^{s_k}$, where $\partial_p$ denotes the  Bockstein homomorphism. 
The operation $Q^I$, or the sequence $I$,  is called {\em admissible} if $ps_j -\varepsilon_j \geq s_{j-1}$ for $ 2 \leq j \leq k$.

An additional convention is that $b(I) = 0$ for $p=2$, and $b(I) = \epsilon_1$ in case $p > 2$.
In addition, the empty sequence $I$ is assumed to be  admissible and satisfies $d(I) = 0$, $\ell(I) = 0$, $e(I) = \infty$, and $b(I) = 0$.

% -------------------------------------------------------------------%
% -------------------------------------------------------------------%
\subsection{The space $Q(X)$}
\label{sec:The space $Q(X)$}
% -------------------------------------------------------------------%
% -------------------------------------------------------------------%

For a pointed compactly generated weak Hausdorff space $X$ we use the following classical notation:
\[
Q(X):=\colim_{d\geq 0}\Omega^d\Sigma^d(X),
\]
where $\Sigma^d(\cdot)$ denotes the $d$-iterated reduced suspension.
The space $Q(X)$ is an $E_{\infty}$-space.

\begin{theorem}
\label{th:Homology_of_Q(X)}
Let $X$ be a path connected weak Hausdorff space with non-degenerate base point.
The homology $H_*(Q(X);\FF_p)$ is isomorphic to the free graded commutative associative Hopf algebra generated by the admissible sequences of Araki--Kudo--Dyer--Lashof operations with excess strictly greater than $q$ acting on a basis of $H_{q}(X;\FF_p)$ for every $q\geq 1$. 
\end{theorem}

\noindent
For simplicity the action described in the theorem will be called ``the action of appropriate admissible sequences of Araki--Kudo--Dyer--Lashof operations on a  basis of $H_{\geq 1}(X;\FF_p)$''.

Next we review some properties of $Q(\cdot)$, for more details consult~\cite{cohen},~\cite{dyer-lahof},~\cite{may533}:

\begin{compactenum}
\item If $f\colon X\longrightarrow Y$ is continuous map of pointed compactly generated
  weak Hausdorff spaces, then the following diagram commutes:
\[
\xymatrix@!C=6em
{
X\ar[r]^-{i_X}\ar[d]^{f} & \Omega^d\Sigma^dX\ar[r]^-{j_{X,d}}\ar[d]^{\Omega^d\Sigma^d(f)} &Q(X)\ar[d]^{Q(f)}\\
Y\ar[r]^-{i_Y}& \Omega^d\Sigma^dY\ar[r]^-{j_{Y,d}}& Q(Y)
}
\]
where $i_X$, $i_Y$, $j_{X,d}$, $j_{Y,d}$ are inclusions, and the maps $\Omega^d\Sigma^d(f)$ and $Q(f)$ are induced by $f$.
More precisely, the map $i_X$ is defined by
\[
x\longmapsto \big(  (S^d,*)\rightarrow S^d\times\{x\} \hookrightarrow \Sigma^dX\big),
\]
where $S^d\times\{x\}$ is a subset of $\Sigma^dX:=S^d\wedge X=S^d\times X/(S^d\times\{x_0\}\cup \{*\}\times X)$ 
and $x_0\in X$ is the base point.

\noindent
Furthermore, the inclusions $i_X$ and $i_Y$ induce injections in homology with any simple
coefficients~\cite[map $\phi_*$ on pp.\,78]{dyer-lahof},~\cite[pp.\,225]{cohen}, while
$j_{X,d}$ and $j_{Y,d}$ induce injections in homology with coefficients in $\FF_p$ when
\begin{compactenum}[\rm (i)]
\item $p=2$ and $X=S^n$ is a sphere, or
\item $p$ is an odd prime and $X=S^n$ is a sphere with $d+n$ being odd.
\end{compactenum}
This is a consequence of~\cite[Thm.\,3.2]{cohen}.

\item If $f\colon X\longrightarrow Y$ induces a surjection in homology with 
  coefficients in $\FF_p$, the same is true for the map $Q(f)\colon Q(X)\longrightarrow Q(Y)$.
  This is a direct consequence of Theorem~\ref{th:Homology_of_Q(X)}.

\item Let $X$ and $Y$ be pointed compactly generated weak Hausdorff spaces with non-degenerate base point. 
  If $X$ and $Y$ are homotopy equivalent as pointed spaces, then $Q(X)\simeq Q(Y)$.

\item If $X$ is a suspension or the homology $H_*(X;\FF_p)$ is primitively generated (as a
  coalgebra; that is each element of the reduced homology is primitive), then the homology
  Hopf algebra $H_*(Q(X);\FF_p)$ is primitively generated (there is a choice of algebra
  generators that are primitive; this does not imply that each element in positive degrees
  is primitive)~\cite[Proof of Corollary after Thm.\,5.1]{dyer-lahof}, or equivalently,
  according to~\cite[Thm.\,4.2]{dyer-lahof}, we have that:
\[
\h(H^*(Q(X);\FF_p))\leq p.
\]
For details in case $p=2$ consult~\cite[Thm.\,7.3]{cohen: course1987}.
\end{compactenum}

\smallskip

The following result that will be used repeatedly is well-known, we give for the reader's convenience a short proof.

\begin{lemma}
\label{lemma:cofib--fib--Q(X)}
If the sequence of pointed finite path connected  cell complexes 
\[
\xymatrix@!
{
X\ar[r]^-{f}&Y\ar[r]^-{g}&Z
}
\]
is a cofibration of pointed spaces, then there is a natural weak homotopy equivalence
\[
Q(X) \overset{i}{\longrightarrow} \hofib(Q(g) \colon Q(Y) \longrightarrow Q(Z)),
\]
where $\hofib$ denotes the homotopy fiber.
\end{lemma}
\begin{proof}
There is a natural isomorphism
\[
\alpha(X) \colon \pi_n^s(X) \overset{\cong}{\longrightarrow} \pi_n(Q(X))
\]
from the stable homotopy of $X$ to the homotopy of $Q(X)$. Since stable homotopy is a
homology theory, we obtain a long exact sequence
\begin{multline*}
\cdots \xrightarrow{\partial^s_{n+1}} \pi_n^s(X) \xrightarrow{\pi_n^s(f)} \pi_n^s(Y) \xrightarrow{\pi_n^s(g)} \pi_n^s(Z) 
\\
\xrightarrow{\partial_n^s}  \pi_{n-1}^s(X) \xrightarrow{\pi_{n-1}^s(f)} \pi_{n-1}^s(Y) 
\xrightarrow{\pi_{n-1}^s(g)} \pi_{n-1}^s(Z) \xrightarrow{\partial_{n-1}^s} \cdots
\end{multline*}
Using the natural isomorphisms $\alpha$, we obtain a long exact sequence
\begin{multline*}
\cdots \xrightarrow{\partial_{n+1}} \pi_n(Q(X)) \xrightarrow{\pi_n(Qf)} \pi_n(Q(Y)) \xrightarrow{\pi_n^s(Q(g))} \pi_n^s(Q(Z)) 
\\
\xrightarrow{\partial_n}  \pi_{n-1}(Q(X)) \xrightarrow{\pi_{n-1}(Q(f))} \pi_{n-1}(Q(Y)) 
\xrightarrow{\pi_{n-1}(Q(g))} \pi_{n-1}^s(Q(Z)) \xrightarrow{\partial_{n-1}} \cdots
\end{multline*}
Since the composite $g \circ f$ is constant and hence $Q(g) \circ Q(f) = Q(g\circ f)$ is
constant there exists a preferred nullhomotopy.
Therefore we get a natural map
\[
i \colon Q(X) \longrightarrow \hofib(Q(g) \colon Q(Y) \longrightarrow Q(Z))
\]
to the homotopy fiber of $Q(g)$. We obtain a commutative diagram with exact rows, where the lower
row is a part of the long exact homotopy sequence of a fibration
{\small
\[
\xymatrix@!C=19mm{
\pi_{n+1}(Q(Y)) \ar[r]^{\pi_{n+1}(Q(g))} \ar[d]^{\id}
&
\pi_{n+1}(Q(Z)) \ar[r]^{\partial_{n+1}} \ar[d]^{\id}
&
\pi_n(Q(X)) \ar[r]^{\pi_n(Q(f))} \ar[d]^{\pi_n(i)}
&
\pi_n(Q(Y)) \ar[r]^{\pi_n(Q(g))} \ar[d]^{\id}
&
\pi_n(Q(Z)) \ar[d]^{\id}
\\
\pi_{n+1}(Q(Y)) \ar[r] ^{\pi_{n+1}(Q(g))}
&
\pi_{n+1}(Q(Z)) \ar[r]^{\delta_{n+1}}
&
\pi_n(\hofib) \ar[r]^{\pi_n(Q(j))}  
&
\pi_n(Q(Y)) \ar[r]^{\pi_n(Q(g))} 
&
\pi_n(Q(Z)) .
}
\]
}
By the Five-Lemma the middle vertical arrow is an isomorphism.

\end{proof}

The second result we are going to use is due to Kahn \& Priddy~\cite[Thm.\,pp.\,103]{kahn - priddy}.

\begin{theorem}\label{theorem:tools-2}
For every prime $p$ there are maps
\[
h_p\colon\Omega^\infty_0 S^\infty\longrightarrow Q(\BB\Sym_p)
\qquad\text{and}\qquad
\theta_p\colon Q(\BB\Sym_p)\longrightarrow\Omega^\infty_0 S^\infty
\]
such that the following composition
\[
\xymatrix@!
{
\Omega^\infty_0 S^\infty\ar[r]^-{h_p}& Q(\BB\Sym_p)\ar[r]^-{\theta_p}&\Omega^\infty_0 S^\infty
}
\] 
is a $p$-local equivalence, i.e., an isomorphism in cohomology with coefficients in $\FF_p$.
\end{theorem}

Next ingredient comes from the work of Cohen, May \& Taylor~\cite[Thm.\,2.7]{cohen - may - taylor},~\cite{cohen - may - taylor - II}.
\begin{proposition}\label{proposition:tools}
Let $d\geq1$, $k\geq1$ be integers, and $p$ be a prime.
Then, there are continuous maps
\[
h_{d,p,k}\colon F(\RR^d,k)/\Sym_k\longrightarrow Q(F(\RR^d,p)/\Sym_p)
\quad\textrm{and}\quad
t_{d,p}\colon \Omega^d_0S^d \longrightarrow Q(F(\RR^d,p)/\Sym_p)
\]
such that the following diagram commutes
\[
\xymatrix@R25pt@C28pt
{
\colim_{k\geq0}F(\RR^d,k)/\Sym_k\ar[rr]^-{\colim_{k\geq0}h_{d,p,k}}\ar@{^{(}->}[dd]\ar[rd]^{\alpha_{d}}&
&
Q(F(\RR^d,p)/\Sym_p)\ar@{^{(}->}[dd]\ar[rdd]
\\
 &  
 \Omega^d_0S^d\ar[ru]^{t_{d,p}}\ar@{^{(}->}[dd] &
\\
\colim_{k\geq0}F(\RR^{\infty},k)/\Sym_k
\ar[rd]^{\alpha_{\infty}}
\ar[rr]^-{\colim_{k\geq0}h_{\infty,p}}|!{[r];[r]}\hole 
& 
&
Q(\BB\Sym_p)\ar[r]^{\theta_p}&\Omega^{\infty}_0 S^{\infty}\\
 &  
 \Omega^{\infty}_0S^{\infty}\ar[ru]^{h_{p}}&
}
\]
where the map $\theta_p$ is introduced in Theorem~\ref{theorem:tools-2}.
\end{proposition}

\begin{corollary}\label{corollary:tools-2}
Let $d\geq1$ be an integer and $p$ be a prime.
If $p=2$, or $p>2$ is an odd prime and $d$ is an odd integer, 
the composition map $h_{d,p,k}=\colim_{k\geq0}h_{d,p,k}\circ i_{d,k}$
\[
F(\RR^d,k)/\Sym_k
\longrightarrow
\colim_{k\geq0}F(\RR^d,k)/\Sym_k
\longrightarrow
Q(F(\RR^d,p)/\Sym_p)
\]
induces an epimorphism in cohomology
\[
h_{d,p,k}^*\colon H^*(Q(F(\RR^d,p)/\Sym_p);\FF_p)\longrightarrow H^*(F(\RR^d,k)/\Sym_k;\FF_p).
\]
\end{corollary}
\begin{proof}
  The proof will be obtained by combining Theorems~\ref{theorem:tools-1} and~\ref{theorem:tools-2} 
  with Proposition~\ref{proposition:tools}.  
  Since the structural map $i_{d,k}\colon F(\RR^d,k)/\Sym_k\longrightarrow\colim_{k\geq0}F(\RR^d,k)/\Sym_k$ induces a 
  monomorphism in homology it is enough to prove that the induced map $(\colim_{k\geq0}h_{d,p,k})_*$ in homology with
  coefficients in $\FF_p$ is a monomorphism.  
  Consider the following commutative diagram:
 \[
\xymatrix@R15pt@C50pt
 {
F(\RR^d,k)/\Sym_k\ar[r]^-{i_{d,k}}&
\colim_{k\geq0}F(\RR^d,k)/\Sym_k\ar[r]^-{\colim_{k\geq0}h_{d,p,k}}\ar[dd]^{\alpha_{d}}\ar[rdd]^{\psi}& Q(F(\RR^d,p)/\Sym_p)\ar@{^{(}->}[d]\\ 
 &  & Q(\BB\Sym_p)\ar[d]^{\theta_p}\\  
 &\Omega_0^dS^d\ar[r]^-{j_d}
 & \Omega_0^{\infty} S^{\infty}
 }
 \]
 where both $i_{d,k}$ and $j_d$ are inclusions from the definition of the colimit.  In
 order to prove that $(\colim_{k\geq0}h_{d,p,k})_*$ is a monomorphism, it suffices to
 prove that $\psi_*$ is a monomorphism. This follows from the fact that
 $\psi=j_d\circ\alpha_{d}$, and both maps $j_d$ and $\alpha_{d}$ induce monomorphisms in
 homology provided that $p=2$ or $p>2$ and $d$ is odd.
\end{proof}

\begin{corollary}\label{corollary:tools-3}
Let $d\geq1$ be an integer and $p$ be a prime.
If $p=2$, or $p>2$ is an odd prime and $d$ is an odd integer, the map then 
$t_{d,p}\colon \Omega^d_0S^d \longrightarrow Q(F(\RR^d,p)/\Sym_p)$
induces an epimorphism in cohomology
\[
t_{d,p}^*\colon H^*(Q(F(\RR^d,p)/\Sym_p);\FF_p)\longrightarrow H^*(\Omega^d_0S^d;\FF_p).
\]
\end{corollary}
\begin{proof}
This is a direct consequence of the previous corollary and the fact that, according to Theorem~\ref{theorem:tools-1}, 
the map $\alpha_{d}$ induces an isomorphism in cohomology with $\FF_p$ coefficients.
\end{proof}

% -------------------------------------------------------------------%
% -------------------------------------------------------------------%
\subsection{Height of $H^*(Q(F(\RR^d,p)/\Sym_p);\FF_p)$}
\label{sec:On the homology and cohomology of unordered configuration space}
% -------------------------------------------------------------------%
% -------------------------------------------------------------------%

First, we recall some know facts, for example consult \cite[Prop.\,5.1(iii) and Thm.\,5.2]{cohen}.

\begin{proposition}\label{prop:cohomology_of_S_p}
  $H^*(\Sym_p;\FF_p)\cong \Lambda[e]\otimes \FF_p[t]$, where $\Lambda(\cdot)$
  denotes the exterior algebra, $e$ is a class of degree $2p-3$ and $t$ is the
  Bockstein of $e$, and so a class of degree $2p-2$.
\end{proposition}

\begin{theorem}\label{theorem:cohomology_of_F(odd,p)}
Let $p$ be an odd prime and $d\ge3$ be an odd integer. 
Then there is an isomorphism of algebras:
\[
H^*(F(\RR^d,p)/\Sym_p;\FF_p)\cong H^{\leq (d-1)(p-1)}(\Sym_p;\FF_p),
\]
that is natural in $d$.
\end{theorem}

The Universal coefficient theorem implies that
$
H_i(F(\RR^d,p)/\Sym_p;\FF_p)\cong \FF_p
$
if and only if $i\in\{2(p-1)\ell : 0\leq \ell\leq \tfrac{d-1}{2}\}\cup\{2(p-1)\ell -1 : 1\leq \ell\leq \tfrac{d-1}{2}\}$.
Let $x_{\ell}$ and $y_{\ell}$, for $1\leq \ell\leq \tfrac{d-1}{2}$, denote additive generators of $H_{\geq 1}(F(\RR^d,p)/\Sym_p;\FF_p)$ where
\[
\deg(x_{\ell})=2(p-1)\ell,\quad\deg(y_{\ell})=2(p-1)\ell-1, \quad \partial^p_*(x_{\ell})=y_{\ell},
\]
where $\partial^p_*$ is the (homology) Bockestein.
In addition, $x_0=1\in H_0(F(\RR^d,p)/\Sym_p;\FF_p)$.

Let $\Delta\colon F(\RR^d,p)/\Sym_p\longrightarrow F(\RR^d,p)/\Sym_p\times F(\RR^d,p)/\Sym_p$ be the diagonal embedding and $\Delta_*$ the induced map in homology.
Then from Theorem~\ref{theorem:cohomology_of_F(odd,p)} follows:
\begin{equation}
\label{eq:Delta(x)}
\Delta_* (x_{\ell})=\sum_{i=0}^{\ell}x_i\otimes x_{\ell-i},
\end{equation}
and
\begin{equation}
\label{eq;Delta(y)}
\Delta_* (y_{\ell})=y_{\ell}\otimes 1 + \sum_{i=1}^{\ell-1}x_i\otimes y_{\ell-i}+\sum_{i=1}^{\ell-1} y_{\ell-i}\otimes x_i+1\otimes y_{\ell}.
\end{equation}

We record a few direct consequences of the previous results.

\begin{corollary}\label{cor : F(R^3,k)}
Let $p$ be an odd prime. 
The cup product in $\widetilde{H}^*(F(\RR^3,p)/\Sym_p;\FF_p)$ is trivial, or dually the homology $H_*(F(\RR^3,p)/\Sym_p;\FF_p)$ is primitively generated as a coalgebra. 
Moreover,
\[
\h(H^*(Q(F(\RR^3,p)/\Sym_p));\FF_p))\leq p.
\]
\end{corollary}
\begin{proof}
From Theorem~\ref{theorem:cohomology_of_F(odd,p)} we have that
\[
H^*(F(\RR^3,p)/\Sym_p;\FF_p)\cong H^{\leq 2p-2}(\Sym_p;\FF_p).
\]
Now Proposition~\ref{prop:cohomology_of_S_p} implies that 
\[
\widetilde{H}^i(F(\RR^3,p)/\Sym_p;\FF_p)\neq 0
\quad\textrm{if and only if}\quad
i\in \{2p-3,2p-2\}.
\]
Thus, the cup product in $\widetilde{H}^*(F(\RR^3,p)/\Sym_p;\FF_p)$ is trivial, or 
$H_*(F(\RR^3,p)/\Sym_p;\FF_p)$ is primitively generated as a coalgebra. 

\noindent
Property (4) of $Q(X)$ implies that $H_*(Q(F(\RR^3,p)/\Sym_p);\FF_p)$ is primitively generated as a Hopf algebra, or equivalently
\[
\h(H^*(Q(F(\RR^3,p)/\Sym_p);\FF_p))\leq p.
\]
\end{proof}

Let $C^d_n$, for $n\leq d$ odd integers, denote the homotopy cofiber of the natural inclusion
\[
F(\RR^n,p)/\Sym_p\longrightarrow F(\RR^d,p)/\Sym_p.
\]
The natural inclusion does not have to be a cofibration.
Therefore, from now on, we assume that it is transformed into a cofibration by substituting the codomain with a mapping cylinder, consult for example~\cite[Section\,3.5]{arkowitz}. 

Since $C^d_n$ is a cofiber of the natural inclusion $F(\RR^n,p)/\Sym_p\longrightarrow F(\RR^d,p)/\Sym_p$ we have that:
\[
H_*(C^d_n,\mathrm{pt};\FF_p)\cong H_*(F(\RR^d,p)/\Sym_p,F(\RR^n,p)/\Sym_p;\FF_p),
\]
and
\[
H^*(C^d_n,\mathrm{pt};\FF_p)\cong H^*(F(\RR^d,p)/\Sym_p,F(\RR^n,p)/\Sym_p;\FF_p).
\]
The long exact sequence of the pair $(F(\RR^d,p)/\Sym_p,F(\RR^n,p)/\Sym_p)$ in (co)homology and Theorem~\ref{theorem:cohomology_of_F(odd,p)} imply
\begin{equation}
\label{eq:homology_of_cofiber}
H_i(F(\RR^d,p)/\Sym_p,F(\RR^n,p)/\Sym_p;\FF_p)\cong H^i(F(\RR^d,p)/\Sym_p,F(\RR^n,p)/\Sym_p;\FF_p) =0,
\end{equation}
for $i\leq(n+1)(p-1)-2$ or $i\geq(d-1)(p-1)+1$.

As a first direct application of the relation \eqref{eq:homology_of_cofiber} we get the following estimate for the height of $H^*(Q(C_n^d);\FF_p)$ in a special case.

\begin{corollary}\label{cor : cofiber}
Let $p$ be an odd prime, and  $d\geq3$, $n\ge3$ be odd integers with $n\leq d$.
Then for $2n+2\geq d$ the cup product in $\widetilde{H}^*(C^d_n;\FF_p)$ is trivial, or dually the homology $H_*(C^d_n;\FF_p)$ is primitively generated as a coalgebra.
Furthermore,
\[
\h(H^*(Q(C_n^d);\FF_p))\leq p.
\]
\end{corollary}
\begin{proof}
Using \eqref{eq:homology_of_cofiber}, and the fact that
\[
2n+2\geq d\qquad\Leftrightarrow\qquad (n+1)(p-1)+(n+1)(p-1)-1\geq(d-1)(p-1)+1,
\]
we conclude that the cup product in $\widetilde{H}^*(C^d_n;\FF_p)$ is trivial, i.e., the cup product of any two non-zero cohomology classes is zero.

\noindent
Again, property (4) of $Q(X)$ implies that $H_*(Q(C_n^d);\FF_p)$ is primitively generated as a Hopf algebra, or equivalently
$
\h(H^*(Q(C_n^d);\FF_p))\leq p
$.

\end{proof}

Now we prove two important technical lemmas.

\begin{lemma}
\label{lemma:hight of Q(cofiber)}
Let $p$ be an odd prime.
If $3\leq d\leq 2p-1$ is odd integer, then
\begin{equation}
\label{eq:h-1}
\h(H^*(Q(F(\RR^{d},p)/\Sym_p);\FF_p))\leq p.
\end{equation}
\end{lemma}

\begin{proof}
Let $3\leq d\leq 2p-1$ and $X:=F(\RR^{d},p)/\Sym_p$. 
The proof will be conducted as follows:
\begin{compactitem}

\item The aim is to prove that $H_*(Q(X);\FF_p)$ is primitively generated as a Hopf algebra, i.e., that there exists a choice of algebra generators that are primitive.
Then \cite[Thm.\,4.2]{dyer-lahof} implies that $\h(H^*(Q(X);\FF_p))\leq p$.

\item The homology $H_*(Q(X);\FF_p)$ is generated by an action of appropriate admissible sequences of Araki--Kudo--Dyer--Lashof operations on a basis of $H_{\geq1}(X;\FF_p)$, Theorem~\ref{th:Homology_of_Q(X)}. 
Recall that Lemma~\ref{lemma:Q(primitve)} states that $Q^s(x)$ of a primitive $x$ is also a primitive for any $s$. 
Thus, it suffices to modify the canonical basis of $H_{\geq1}(X;\FF_p)$, now seen as a part of $H_*(Q(X);\FF_p)$, in such a way that all elements of the new basis are primitives. 

\item Modification of each non-primitive element in the canonical basis  will be done using the Newton polynomials and adapting the argument from \cite[Def.\,5.1 and Lemma\,5.2]{willington}.
\end{compactitem}	

\medskip
\noindent
Let $x_{\ell}$ and $y_{\ell}$, $1\leq \ell\leq \tfrac{d-1}{2}$, denote the additive generators of $H_{\geq 1}(X;\FF_p)$ where
\[
\deg(x_{\ell})=2(p-1)\ell,\quad\deg(y_{\ell})=2(p-1)\ell-1, \quad \partial^p_*(x_{\ell})=y_{\ell},
\]
and $\partial^p_*$ is the (homology) Bockestein.
Denote $x_0=1\in H_{0}(F(\RR^d,p)/\Sym_p;\FF_p)$.

\noindent
As we have seen: 
\[
\Delta_* (x_{\ell})=\sum_{i=0}^{\ell}x_i\otimes x_{\ell-i},\quad
\Delta_* (y_{\ell})=y_{\ell}\otimes 1 + \sum_{i=1}^{\ell-1}(x_i\otimes y_{\ell-i}+ y_{\ell-i}\otimes x_i)+1\otimes y_{\ell},
\]
where $\Delta:=\Delta_2$ is the diagonal embedding and $\Delta_*$ the induced map in homology.

\noindent
Let us define a sequence of polynomials associated to the even dimensional generators. 
The elements we define will live in $H_*(Q(X);\FF_p)$.
Let $v_1:=x_1$. 
Then for $1\leq \ell\leq \tfrac{d-1}{2}$ we define a sequence of elements by: 
\[
v_{\ell}:=\ell x_{\ell} - \sum_{i=1}^{\ell-1}x_iv_{\ell-i},
\]
where multiplication is the Pontryagin product.
The elements $v_{\ell}$ are called {\em Newton polynomials}.
The assumption $d\leq 2p-1$ implies that $\ell\leq \tfrac{d-1}{2}\leq p-1$, meaning that $(\ell,p)=1$ for all $1\leq \ell\leq \tfrac{d-1}{2}$. 
Consequently, $\tfrac{1}{\ell}$ is invertible in $\FF_p$. 

\begin{claim}
$v_1,\ldots,v_{\tfrac{d-1}{2}}$ are primitives in $H_*(Q(X);\FF_p)$.
\end{claim}
\begin{proof}
The class $v_1=x_1$ is of degree $2(p-1)$ and is primitive. 
The class $x_2$ is not primitive since $\Delta_*(x_2)= x_2\otimes 1 + x_1\otimes x_1 +1\otimes x_2$. 
We verify that $v_2=2x_2-x_1^2$ is primitive:
{\small
\begin{eqnarray*}
	\Delta_*(v_2) &=& 2\Delta_*(x_2)-\Delta_*(x_1)^2=2(x_2\otimes 1 + x_1\otimes x_1 +1\otimes x_2) - (x_1\otimes 1 + 1\otimes x_1)^2\\
	&=& 2(x_2\otimes 1) + \cancel{2(x_1\otimes x_1)} +2(1\otimes x_2) - x_1^2\otimes 1 - \cancel{2( x_1\otimes x_1)} - 1\otimes x_1^2\\
	&=& (2x_2-x_1^2)\otimes 1 + 1\otimes (2x_2-x_1^2).
\end{eqnarray*}}For the general case assume that $3\leq \ell\leq \tfrac{d-1}{2}$, that $v_1,\ldots,v_{\ell-1}$ are primitives, and consider how $\Delta_*$ acts on $v_{\ell}=\ell x_{\ell} - \sum_{i=1}^{\ell-1}x_iv_{\ell-i}$.
Since $\Delta_*$ is a homomorphism with respect to Pontryagin product and $v_1,\ldots,v_{\ell-1}$ are primitives we have that:
{\small
\begin{eqnarray*}
\Delta_*(v_{\ell}) &=& \ell\Delta_*(x_{\ell})-\sum_{i=1}^{\ell-1}\Delta_*(x_iv_{\ell-i})=
\ell\sum_{i=0}^{\ell}x_i\otimes x_{\ell-i}-\sum_{i=1}^{\ell-1}\Delta_*(x_i)\Delta_*(v_{\ell-i})\\
&=&\ell\sum_{i=0}^{\ell}x_i\otimes x_{\ell-i} - \sum_{i=1}^{\ell-1}
\Big( \sum_{j=0}^{i}  x_j\otimes x_{i-j}   \Big)(v_{\ell-i}\otimes 1+ 1\otimes v_{\ell-i} )\\
&=&\ell( x_{\ell}\otimes 1+1\otimes x_{\ell}) + \ell\sum_{i=1}^{\ell-1}x_i\otimes x_{\ell-i}\\
& &- \sum_{i=1}^{\ell-1}
\Big( x_i\otimes 1+\sum_{j=1}^{i-1}  x_j\otimes x_{i-j} +1\otimes x_i   \Big)(v_{\ell-i}\otimes 1+ 1\otimes v_{\ell-i} )\\
&=&\ell( x_{\ell}\otimes 1+1\otimes x_{\ell}) 
- \sum_{i=1}^{\ell-1}x_iv_{\ell-i}\otimes 1 - \sum_{i=1}^{\ell-1}1\otimes x_iv_{\ell-i}\\
& &+\ell\sum_{i=1}^{\ell-1}x_i\otimes x_{\ell-i}
- \sum_{i=1}^{\ell-1}x_i\otimes v_{\ell-i}- \sum_{i=1}^{\ell-1}v_{\ell-i}\otimes x_i\\
& &-\sum_{i=2}^{\ell-1}\sum_{j=1}^{i-1}  x_jv_{\ell-i}\otimes x_{i-j} 
-\sum_{i=2}^{\ell-1}\sum_{j=1}^{i-1}  x_j\otimes x_{i-j} v_{\ell-i}.
\end{eqnarray*}	}Therefore, in order to prove that $v_{\ell}$ is primitive it remains to show that:
{\small
\begin{eqnarray*}
\ell\sum_{i=1}^{\ell-1}x_i\otimes x_{\ell-i}
- \sum_{i=1}^{\ell-1}x_i\otimes v_{\ell-i}- \sum_{i=1}^{\ell-1}v_{\ell-i}\otimes x_i & &\\
-\sum_{i=2}^{\ell-1}\sum_{j=1}^{i-1}  x_jv_{\ell-i}\otimes x_{i-j} 
-\sum_{i=2}^{\ell-1}\sum_{j=1}^{i-1}  x_j\otimes x_{i-j} v_{\ell-i} &=&0.
\end{eqnarray*}}
Indeed, careful manipulations and several reindexing give us that:
{\small
\begin{eqnarray*}
\ell\sum_{i=1}^{\ell-1}x_i\otimes x_{\ell-i}
- \sum_{i=1}^{\ell-1}x_i\otimes v_{\ell-i}- \sum_{i=1}^{\ell-1}v_{\ell-i}\otimes x_i & &\\
-\sum_{i=2}^{\ell-1}\sum_{j=1}^{i-1}  x_jv_{\ell-i}\otimes x_{i-j} 
-\sum_{i=2}^{\ell-1}\sum_{j=1}^{i-1}  x_j\otimes x_{i-j} v_{\ell-i} &=&\\
x_{\ell-1}\otimes(\ell x_1-v_1)+\sum_{i=1}^{\ell-2}x_i\otimes (\ell x_{\ell-i}-v_{\ell-i})- v_1\otimes x_{\ell-1}-\sum_{i=1}^{\ell-2}v_{\ell-i}\otimes x_i& &\\
-\sum_{i=2}^{\ell-1}\sum_{j=1}^{i-1}  x_jv_{\ell-i}\otimes x_{i-j} 
-\sum_{j=1}^{\ell-2}x_j\otimes \Big(\sum_{r=1}^{\ell-j-1} x_{r} v_{\ell-r-j}\Big)&=&
\end{eqnarray*}
\begin{eqnarray*}
x_{\ell-1}\otimes \cancel{(\ell x_1-v_1-(\ell-1)x_1)}+
x_1\otimes\cancel{\Big(\ell x_{\ell-1}-v_{\ell-1}-\sum_{r=1}^{\ell-2} x_{r} v_{\ell-r-1}-x_{\ell-1}\Big)}& &\\
+\sum_{i=2}^{\ell-2}x_i\otimes \cancel{\Big(\ell x_{\ell-i}-v_{\ell-i}-\sum_{r=1}^{\ell-i-1} x_{r} v_{\ell-r-i}-ix_{\ell-i}\Big)}& &\\
\sum_{i=2}^{\ell-1}\Big(\sum_{r=1}^{i-1}x_rv_{i-r}\Big)\otimes x_{\ell-i}-\sum_{i=2}^{\ell-1}\sum_{j=1}^{i-1}  x_jv_{\ell-i}\otimes x_{i-j}&=&0.
\end{eqnarray*}}
Thus, $v_1,\ldots,v_{\tfrac{d-1}{2}}$ are primitives.
\end{proof}

\noindent
A direct consequence of the previous claim is the following fact.

\begin{claim}
$\partial^p_*v_1,\ldots,\partial^p_*v_{\tfrac{d-1}{2}}$ are primitives in $H_*(Q(X);\FF_p)$, where $\partial^p_*$ denotes the homology Bockestein.
\end{claim}
\begin{proof}
By the previous claim each $v_{\ell}$ for $1\leq\ell\leq \tfrac{d-1}{2}$ is primitive, i.e.,
$
\Delta_*(v_{\ell})=v_{\ell}\otimes 1+1\otimes v_{\ell}
$.
Since the Bockestein homomorphism is natural:
\[
\Delta_*(\partial^p_*v_{\ell})=\partial^p_* \Delta_*(v_{\ell}) =\partial^p_*(v_{\ell}\otimes
 1+1\otimes v_{\ell}) = (\partial^p_*v_{\ell})\otimes 1+1\otimes (\partial^p_*v_{\ell}).
\]
Thus, $\partial^p_*v_1,\ldots,\partial^p_*v_{\tfrac{d-1}{2}}$ are primitives.
\end{proof}

\noindent
To conclude the proof of part (1) of the lemma we need to verify the following fact.

\begin{claim}
The appropriate admissible sequence of Araki--Kudo--Dyer--Lashof operations acting on 
$\big\{v_1,\ldots,v_{\tfrac{d-1}{2}},\partial^p_*v_1,\ldots,\partial^p_*v_{\tfrac{d-1}{2}}\big\}$ generate $H_*(Q(X);\FF_p)$ as a Hopf algebra.
\end{claim}
\begin{proof}
For a set $U$ we denote by $\mathrm{span}(U)$ the commutative associative Hopf algebra generated by appropriate admissible Araki--Kudo--Dyer--Lashof operations acting on $U$.
Consider for $1\leq\ell\leq\tfrac{d-1}{2}$ the following sets:
\[
S_{\ell}=\{ x_1,\ldots,x_{\ell}\},\ T_{\ell}=\{ y_1,\ldots,y_{\ell}\}, \ A_{\ell}=\{ v_1,\ldots,v_{\ell}\},\ B_{\ell}=\{ \partial^p_*v_1,\ldots,\partial^p_*v_{\ell}\}.
\]
We need to prove that:
\[
\mathrm{span}\Big(S_{\tfrac{d-1}{2}}\cup T_{\tfrac{d-1}{2}}\Big)=\mathrm{span}\Big(A_{\tfrac{d-1}{2}}\cup B_{\tfrac{d-1}{2}}\Big).
\]
The proof proceeds by induction on $\ell$. 
For $\ell=1$, by definition of the Newton polynomials, $v_1=x_1$ and $\partial^p_*v_1=\partial^p_*x_1=y_1$. 
Thus $S_1\cup T_1=A_1\cup B_1$.
Assume that $\mathrm{span}(S_{i}\cup T_{i})=\mathrm{span}(A_{i}\cup B_{i})$ for $1\leq i\leq \ell-1$.
Since 
\begin{compactitem}
\item $x_{\ell}\notin \mathrm{span}(S_{\ell-1}\cup T_{\ell-1})=\mathrm{span}(A_{\ell-1}\cup B_{\ell-1})$,
\item $v_{\ell}=\ell x_l + z$ where $z\in \mathrm{span}(S_{\ell-1}\cup T_{\ell-1})=\mathrm{span}(A_{\ell-1}\cup B_{\ell-1})$, and
\item $(\ell,p)=1$
\end{compactitem}
we have that $\mathrm{span}(S_{\ell}\cup T_{\ell})=\mathrm{span}(A_{\ell}\cup B_{\ell})$.
This completes the induction.
\end{proof}
Since we have obtained a set of primitive generators for $H_*(Q(X);\FF_p)$ as a Hopf algebra the proof of the lemma is concluded.
\end{proof}

\begin{lemma}
\label{lemma:hight of Q(cofiber)-2}
Let $p$ be an odd prime.
If $d\geq 3$ and $n\geq3$ are odd integers with $n<d<(n+1)p$, then
\begin{equation}
\label{eq:h-2}
\h(H^*(Q(C^d_n);\FF_p))\leq p.
\end{equation}
\end{lemma}
\begin{proof}
The proof of this lemma is similar to the proof of previous one but more involved. 
Let $d,n\geq3$ be odd integers, $n<d<(n+1)p$, and $X:=C^d_n$.
Recall that for $i\geq1$
\begin{multline*}
H_{i}(C^d_n;\FF_p)\cong H_{i}(F(\RR^d,p)/\Sym_p, F(\RR^n,p)/\Sym_p;\FF_p)=\\
\left\{
\begin{array}{ll}
\FF_p,  & i=2(p-1)\ell\text{ \ where \ }\tfrac{n+1}{2}\leq\ell\leq\tfrac{d-1}{2},\\
\FF_p,  & i=2(p-1)\ell-1\text{ \ where \ }\tfrac{n+1}{2}\leq\ell\leq\tfrac{d-1}{2},\\
0,      & \text{otherwise}.
\end{array}
\right.
\end{multline*}
Let $x_{\ell}$ and $y_{\ell}$, $1\leq \ell\leq \tfrac{d-1}{2}$, be additive generators of $H_{\geq 1}(X;\FF_p)$ where
\[
\deg(x_{\ell})=2(p-1)\ell,\quad\deg(y_{\ell})=2(p-1)\ell-1, \quad \partial^p_*(x_{\ell})=y_{\ell},
\]
$x_0=1\in H_{0}(F(\RR^d,p)/\Sym_p;\FF_p)$, $x_{\ell}=0$ and $y_{\ell}=0$ when $1\leq\ell\leq\tfrac{n-1}{2}$ or $\ell\geq\tfrac{d+1}{2}$.
Then we have the same formula for the map $\Delta_*$ induced by diagonal embedding
\[
\Delta_* (x_{\ell})=\sum_{i=0}^{\ell}x_i\otimes x_{\ell-i},\quad
\Delta_* (y_{\ell})=y_{\ell}\otimes 1 + \sum_{i=1}^{\ell-1}(x_i\otimes y_{\ell-i}+ y_{\ell-i}\otimes x_i)+1\otimes y_{\ell}.
\]

The proof is conducted in the same fashion as the proof of the previous lemma:  
\begin{compactitem}

\item It suffices to prove that $H_*(Q(X);\FF_p)$ is primitively generated as a Hopf algebra, i.e., that there exists a choice of algebra generators that are primitive.

\item The homology $H_*(Q(X);\FF_p)$ is generated by an action of appropriate admissible sequences of Araki--Kudo--Dyer--Lashof operations on a basis of $H_{\geq1}(X;\FF_p)$.
Since $Q^s(x)$ of a primitive $x$ is also a primitive for any $s$, we modify the basis $\{x_{\ell},y_{\ell}:\ell\geq0\}$ of $H_{\geq1}(X;\FF_p)$, now seen as a part of $H_*(Q(X);\FF_p)$, in such a way that all elements of the new basis are primitives. 

\item Modification of each non-primitive element in the canonical basis  will be done using extended Newton polynomials that are introduced.
\end{compactitem}

Define the sequence of extended Newton polynomials, associated to even dimensional generators $\{x_{\ell}:0\leq\ell\leq\tfrac{d-1}{2}\}$, as follows:
\begin{equation}
\label{eq:def_extended_newton_poly}
v_{\ell}=\lfloor\tfrac{2\ell}{n+1}\rfloor x_{\ell}-\sum_{i=1}^{\ell-1}x_iv_{\ell-i}.
\end{equation}
Observe that $v_{\ell}\neq 0$ if and only if $\tfrac{n+1}{2}\leq\ell\leq\tfrac{d-1}{2}$, while $x_{\ell}\neq 0$ if and only if $\tfrac{n+1}{2}\leq\ell\leq\tfrac{d-1}{2}$ or $\ell=0$.

\begin{claim}
$\{v_{\ell}: \tfrac{n+1}{2}\leq\ell\leq\tfrac{d-1}{2}\}$ are primitives in $H_*(Q(X);\FF_p)$.
\end{claim}
\begin{proof}
The proof is by induction on $\ell$.
Note that
\begin{equation}
\label{eq:primitives}
v_{\tfrac{n+1}{2}}=x_{\tfrac{n+1}{2}},\ v_{\tfrac{n+3}{2}}=x_{\tfrac{n+3}{2}}, \ \ldots , \ v_{\tfrac{2n}{2}}=x_{\tfrac{2n}{2}}
\end{equation}
are primitives. 
 
Let $\tfrac{2n+2}{2}\leq\ell\leq\tfrac{d-1}{2}$ and assume that all $v_{\tfrac{n+1}{2}},\ldots, v_{\ell-1}$ are primitives.
Consider how $\Delta_*$ acts on $v_{\ell}$:

{\small
\begin{eqnarray*}
\Delta_*(v_{\ell}) &=& \lfloor\tfrac{2\ell}{n+1}\rfloor\Delta_*(x_{\ell})-\sum_{i=1}^{\ell-1}\Delta_*(x_iv_{\ell-i})\\
&=&
\lfloor\tfrac{2\ell}{n+1}\rfloor\sum_{i=0}^{\ell}x_i\otimes x_{\ell-i}-\sum_{i=1}^{\ell-1}\Delta_*(x_i)\Delta_*(v_{\ell-i})\\
&=&\lfloor\tfrac{2\ell}{n+1}\rfloor\sum_{i=0}^{\ell}x_i\otimes x_{\ell-i} - \sum_{i=1}^{\ell-1}
\Big( \sum_{j=0}^{i}  x_j\otimes x_{i-j}   \Big)(v_{\ell-i}\otimes 1+ 1\otimes v_{\ell-i} )\\
&=&\lfloor\tfrac{2\ell}{n+1}\rfloor( x_{\ell}\otimes 1+1\otimes x_{\ell}) + \lfloor\tfrac{2\ell}{n+1}\rfloor\sum_{i=1}^{\ell-1}x_i\otimes x_{\ell-i}\\
& &- \sum_{i=1}^{\ell-1}
\Big( x_i\otimes 1+\sum_{j=1}^{i-1}  x_j\otimes x_{i-j} +1\otimes x_i   \Big)(v_{\ell-i}\otimes 1+ 1\otimes v_{\ell-i} )\\
&=&\lfloor\tfrac{2\ell}{n+1}\rfloor( x_{\ell}\otimes 1+1\otimes x_{\ell}) 
- \sum_{i=1}^{\ell-1}x_iv_{\ell-i}\otimes 1 - \sum_{i=1}^{\ell-1}1\otimes x_iv_{\ell-i}\\
& &+\lfloor\tfrac{2\ell}{n+1}\rfloor\sum_{i=1}^{\ell-1}x_i\otimes x_{\ell-i}
- \sum_{i=1}^{\ell-1}x_i\otimes v_{\ell-i}- \sum_{i=1}^{\ell-1}v_{\ell-i}\otimes x_i\\
& &-\sum_{i=2}^{\ell-1}\sum_{j=1}^{i-1}  x_jv_{\ell-i}\otimes x_{i-j} 
-\sum_{i=2}^{\ell-1}\sum_{j=1}^{i-1}  x_j\otimes x_{i-j} v_{\ell-i}.
\end{eqnarray*}	}

\noindent
Therefore, in order to prove that $v_{\ell}$ is primitive it remains to show that:
{\small
\begin{eqnarray*}
\lfloor\tfrac{2\ell}{n+1}\rfloor\sum_{i=1}^{\ell-1}x_i\otimes x_{\ell-i}
- \sum_{i=1}^{\ell-1}x_i\otimes v_{\ell-i}- \sum_{i=1}^{\ell-1}v_{\ell-i}\otimes x_i & &\\
-\sum_{i=2}^{\ell-1}\sum_{j=1}^{i-1}  x_jv_{\ell-i}\otimes x_{i-j} 
-\sum_{i=2}^{\ell-1}\sum_{j=1}^{i-1}  x_j\otimes x_{i-j} v_{\ell-i} &=&0.
\end{eqnarray*}}

\noindent
Recall that $v_{\ell}\neq 0$ if and only if $\tfrac{n+1}{2}\leq\ell\leq\tfrac{d-1}{2}$, while $x_{\ell}\neq 0$ if and only if $\tfrac{n+1}{2}\leq\ell\leq\tfrac{d-1}{2}$ or $i=0$.
Using equally \eqref{eq:primitives} in the case $\tfrac{2n+2}{2}\leq\ell\leq\tfrac{3(n+1)}{2}-1$
{\small
\begin{eqnarray*}
\lfloor\tfrac{2\ell}{n+1}\rfloor\sum_{i=1}^{\ell-1}x_i\otimes x_{\ell-i}
- \sum_{i=1}^{\ell-1}x_i\otimes v_{\ell-i}- \sum_{i=1}^{\ell-1}v_{\ell-i}\otimes x_i & &\\
-\sum_{i=2}^{\ell-1}\sum_{j=1}^{i-1}  x_jv_{\ell-i}\otimes x_{i-j} 
-\sum_{i=2}^{\ell-1}\sum_{j=1}^{i-1}  x_j\otimes x_{i-j} v_{\ell-i} &=&\\
2\sum_{i=\tfrac{n+1}{2}}^{\ell-\tfrac{n+1}{2}}x_i\otimes x_{\ell-i}
- \sum_{i=\tfrac{n+1}{2}}^{\ell-\tfrac{n+1}{2}}x_i\otimes v_{\ell-i}
- \sum_{i=\tfrac{n+1}{2}}^{\ell-\tfrac{n+1}{2}}v_{\ell-i}\otimes x_i 
-0 -0 &=&0.
\end{eqnarray*}}

\noindent
Next we assume that $\tfrac{3(n+1)}{2}\leq\ell\leq\tfrac{d-1}{2}$. 
After appropriate reindexing of the sums:
{\small
\begin{eqnarray*}
\lfloor\tfrac{2\ell}{n+1}\rfloor\sum_{i=1}^{\ell-1}x_i\otimes x_{\ell-i}
- \sum_{i=1}^{\ell-1}x_i\otimes v_{\ell-i}- \sum_{i=1}^{\ell-1}v_{\ell-i}\otimes x_i & &\\
-\sum_{i=2}^{\ell-1}\sum_{j=1}^{i-1}  x_jv_{\ell-i}\otimes x_{i-j} 
-\sum_{i=2}^{\ell-1}\sum_{j=1}^{i-1}  x_j\otimes x_{i-j} v_{\ell-i} &=&
\end{eqnarray*}
\begin{eqnarray*}
\lfloor\tfrac{2\ell}{n+1}\rfloor\sum_{i=\tfrac{n+1}{2}}^{\ell-\tfrac{n+1}{2}}x_i\otimes x_{\ell-i}
- \sum_{i=\tfrac{n+1}{2}}^{\ell-\tfrac{n+1}{2}}x_i\otimes v_{\ell-i} 
- \sum_{i=\tfrac{n+1}{2}}^{\ell-\tfrac{n+1}{2}}
\big(\lfloor\tfrac{2i}{n+1}\rfloor x_{i}-\sum_{j=1}^{i-1}x_jv_{i-j}
\big)
\otimes x_{\ell-i}& &\\
-\sum_{i=2}^{\ell-1}\sum_{j=1}^{i-1}  x_jv_{\ell-i}\otimes x_{i-j} 
-\sum_{j=\tfrac{n+1}{2}}^{\ell-\tfrac{n+1}{2}}x_j\otimes \Big(\sum_{r=1}^{\ell-j-1} x_{r} v_{\ell-r-j}\Big) &=&\\
\sum_{i=\tfrac{n+1}{2}}^{\ell-\tfrac{n+1}{2}}x_i\otimes
\cancel{\Big(
\lfloor\tfrac{2\ell}{n+1}\rfloor x_{\ell-i}-v_{\ell-i}-\lfloor\tfrac{2i}{n+1}\rfloor x_{\ell-i}-\sum_{r=1}^{\ell-j-1} x_{r} v_{\ell-r-j}
\Big)}&+&\\
\sum_{i=\tfrac{n+1}{2}}^{\ell-\tfrac{n+1}{2}}\sum_{j=1}^{i-1}x_jv_{i-j}\otimes x_{\ell-i}
-\sum_{i=2}^{\ell-1}\sum_{j=1}^{i-1}  x_jv_{\ell-i}\otimes x_{i-j} &=&0.
\end{eqnarray*}}

\noindent
Thus we have proved that $\{v_{\ell}: \tfrac{n+1}{2}\leq\ell\leq\tfrac{d-1}{2}\}$ are primitives in $H_*(Q(X);\FF_p)$.
\end{proof}

\noindent
The following two claims, are established in the same way as in the proof of Lemma~\ref{lemma:hight of Q(cofiber)}.
\begin{claim}
$\{\partial^p_*v_{\ell}: \tfrac{n+1}{2}\leq\ell\leq\tfrac{d-1}{2}\}$
are primitives in $H_*(Q(X);\FF_p)$, where $\partial^p_*$ denotes the homology Bockestein.
\end{claim}

\begin{claim}
The appropriate admissible sequences of Araki--Kudo--Dyer--Lashof operations acting on 
$\{\partial^p_*v_{\ell},v_{\ell}: \tfrac{n+1}{2}\leq\ell\leq\tfrac{d-1}{2}\}$
 generate $H_*(Q(X);\FF_p)$ as a Hopf algebra.
\end{claim}

\noindent
Thus, we have proved that $\h(H^*(Q(C^d_n);\FF_p))\leq p$.

\end{proof}

% -------------------------------------------------------------------%
% -------------------------------------------------------------------%
\subsection{Height comparison lemma for Hopf algebras}
\label{sec:Height comparison lemma for Hopf algebras}
% -------------------------------------------------------------------%
% -------------------------------------------------------------------%

Next ingredient we need is the following comparison lemma.
\begin{lemma}\label{lemma:tools}
Let
\[
\xymatrix@!
{
1\ar[r]& A\ar[r]^-{\alpha}& B\ar[r]^-{\beta}& C\ar[r]&1
}
\]
be a short exact sequence of connected, cocommutative and coassociative, commutative and
associative Hopf algebras over the field $\FF_p$.  Then
\[
\h(B)\leq \h(A)\h(C).
\]
\end{lemma}
\begin{proof}
According to the Borel structural theorem~\cite[Prop.\,7.8 and Thm.\,7.11]{milnor - moore} 
we know that the heights of algebras $A$, $B$ and $C$
are powers of prime or infinity. 
If one of the heights $\h(A)$ or $\h(C)$ are infinity then there is nothing to prove.
Thus, assume that $\h(A)=p^n$ and $\h(C)=p^m$ for some integers $n$ and $m$.
Consider an arbitrary element $b\in B{\setminus}B^*$ and denote by $c:=\beta(b)$.
Then
\[
 \beta(b^{\h(C)})=c^{\h(C)}=0.
\]
Consequently $b^{\h(C)}$ belongs to the two sided ideal generated by $\alpha(A)$ which means that
\[
 b^{\h(C)}=\sum_{i\in I}a_ib_i
\]
for some $a_i\in\alpha(A)$ and $b_i\in B$.
Now using the fact that $\h(A)=p^n$ we have that
\[
 \big(b^{\h(C)}\big)^{\h(A)}=\Big(\sum_{i\in I}a_ib_i\Big)^{\h(A)}=\Big(\sum_{i\in I}a_ib_i\Big)^{p^n}=\sum_{i\in I}(a_i)^{p^n}(b_i)^{p^n}=0.
\]
Thus, $\h(b)\leq \h(A)\h(C)$ for every $b\in B{\setminus}B^*$ and consequently $\h(B)\leq \h(A)\h(C)$.
\end{proof}

%%%%%%%%%%%%%%%%%%%%%%%%%%%%%%%%%%%%%%%%%%%%%%%%%%%%%%%%%%%%%%%%%%%%%%%%%%%%%%%%%%%%%
%%%%%%%%%%%%%%%%%%%%%%%%%%%%%%%%%%%%%%%%%%%%%%%%%%%%%%%%%%%%%%%%%%%%%%%%%%%%%%%%%%%%%
\section{Heights in the cohomology of the unordered configuration space}
%%%%%%%%%%%%%%%%%%%%%%%%%%%%%%%%%%%%%%%%%%%%%%%%%%%%%%%%%%%%%%%%%%%%%%%%%%%%%%%%%%%%%
%%%%%%%%%%%%%%%%%%%%%%%%%%%%%%%%%%%%%%%%%%%%%%%%%%%%%%%%%%%%%%%%%%%%%%%%%%%%%%%%%%%%%

In this section we give estimates for the heights of elements in the 
cohomology of the unordered configuration space $H^*(F(\RR^d,k)/\Sym_k;\FF_p)$.

\begin{theorem}
\label{theorem:hight_bound_for_prime_2}
Let $d\geq 2$ and $k\geq 2$ be integers.
Then
\[
\h(H^{*}(F(\RR^d,k)/\Sym_k;\FF_2))\leq \min\{ 2^t: 2^t\geq d \}.
\]
\end{theorem}

\noindent
In particular, for $d$ a power of two, the theorem answers a conjecture of Vassiliev
\cite[Conj.\,2, pp.\,210]{vassiliev92},~\cite[Conj.\,2, pp.\,75]{vassiliev92book}.
Moreover, by \cite[Proof of Lemma\,2.15, relation (2)]{blagojevic-luck-ziegler-2} the result in the case when both $d$ and $k$ are powers of $2$ is tight,
\[
\h(H^{*}(F(\RR^{2^t},2^m)/\Sym_{2^m};\FF_2))=2^t.
\]

For the case of odd primes a similar result holds.

\begin{theorem}
\label{theorem:hight_bound_for_odd}
Let $d\geq 2$ and $k\geq 2$ be integers, and let $p$ be an odd prime. 
Then
\[
 \h(H^*(F(\RR^d,k)/\Sym_k;\FF_p))\leq\min\{ p^t: 2p^t\geq d \}.
\]
\end{theorem}

\noindent
For an odd prime $p$, $d=2p^t$ and $k=p^m$ Theorem~\ref{theorem : Chern classes} implies that the result is tight:
\[
 \h(H^*(F(\RR^{2p^t},p^m)/\Sym_{p^m};\FF_p))=p^t.
\]

% -------------------------------------------------------------------%
% -------------------------------------------------------------------%
\subsection{Proof of Theorem~\ref{theorem:hight_bound_for_prime_2}}
% -------------------------------------------------------------------%
% -------------------------------------------------------------------%

In order to estimate the heights of the elements in $H^*(F(\RR^d,k)/\Sym_k;\FF_2)$ we use
Corollary~\ref{corollary:tools-2} and estimate the heights of elements in
\[
H^*(Q(F(\RR^d,2)/\Sym_2);\FF_2)
\]
instead.
The advantage of considering $H^*(Q(F(\RR^d,2)/\Sym_2);\FF_2)$ lays in the fact that it has a Hopf algebra 
structure and is subject to the Borel structural theorem 
\cite[Prop.\,7.8 and Thm.\,7.11]{milnor - moore}. 
In particular, this means that
\[
H^*(Q(F(\RR^d,2)/\Sym_2);\FF_2)\cong\bigotimes_{i\in I}\FF_2[x_i]/\langle {x_i}^{s_i}\rangle
\]
where each height $s_i$ is either power of $2$ or $\infty$ for some index set $I$. 

Since $F(\RR^d,2)/\Sym_2\simeq \RP^{d-1}$ and operation $Q(\cdot)$ preserves the homotopy type 
we need to estimate height of the following algebra
\[
 H^*(Q(\RP^{d-1});\FF_2).
\]

Now denote by $\RP^{t}_s=\RP^t/\RP^{s-1}$, for $1\leq s\leq t$, the truncated real projective space 
with the top cell in dimension $t$ and the bottom cell in dimension $s$.
Consider the following cofibrations:
\begin{equation}\label{eq:cofibration-01}
 \xymatrix@!
{
\RP^{2^{t-1}}\ar[r]& \RP^{2^{t}}\ar[r]& \RP^{2^{t}}_{2^{t-1}+1},
}
\end{equation}
\begin{equation}\label{eq:cofibration-02}
 \xymatrix@!
{
\RP^{2^{t-1}-1}\ar[r]& \RP^{2^{t}-1}\ar[r]& \RP^{2^{t}-1}_{2^{t-1}},
}
\end{equation}
\begin{equation}\label{eq:cofibration-03}
 \xymatrix@!
{
\RP^{2^{t-1}-1}\ar[r]& \RP^{2^{t}-j}\ar[r]& \RP^{2^{t}-j}_{2^{t-1}}
}
\end{equation}
where $1\leq j\leq 2^{t-1}+1$.
The cofibration~\eqref{eq:cofibration-03} for $j=1$ coincides with the cofibration~\eqref{eq:cofibration-02}.
Now, applying $Q(\cdot)$ yields the following fibrations (up to homotopy):
\begin{equation}\label{eq:fibration-01}
 \xymatrix@!
{
Q(\RP^{2^{t-1}})\ar[r]& Q(\RP^{2^{t}})\ar[r]& Q(\RP^{2^{t}}_{2^{t-1}+1}),
}
\end{equation}
\begin{equation}\label{eq:fibration-02}
 \xymatrix@!
{
Q(\RP^{2^{t-1}-1})\ar[r]& Q(\RP^{2^{t}-1})\ar[r]& Q(\RP^{2^{t}-1}_{2^{t-1}}),
}
\end{equation}
\begin{equation}\label{eq:fibration-03}
 \xymatrix@!
{
Q(\RP^{2^{t-1}-1})\ar[r]& Q(\RP^{2^{t}-j})\ar[r]& Q(\RP^{2^{t}-j}_{2^{t-1}}).
}
\end{equation}
Homology and cohomology Serre spectral sequences, with coefficients in the field $\FF_2$,
associated to the fibrations~\eqref{eq:fibration-01},
\eqref{eq:fibration-02} and~\eqref{eq:fibration-03} collapse at the
$E_2$-term.  Indeed, for the homology spectral sequence:
\begin{compactitem}[$\bullet$]

\item $Q(\RP^{2^{t}}_{2^{t-1}+1})$, $Q(\RP^{2^{t}-1}_{2^{t-1}})$ and $Q(\RP^{2^{t}-j}_{2^{t-1}}) $ are simply connected, 
and consequently $E_2$-term is given by the tensor product of the (co)homology of the fiber and base (co)homology;

\item these fibrations are multiplicative, and therefore the Leibniz rule holds for the differentials in the homology spectral sequence,~\cite{felix-halperin-thomas}. Consequently, the spectral sequences are completely determined 
by the differential arising from the zero row, and finally
\item homology surjections induced by the following quotient maps
         \[
         \RP^{2^{t}}\longrightarrow\RP^{2^{t}}_{2^{t-1}+1},\quad
         \RP^{2^{t}-1} \longrightarrow\RP^{2^{t}-1}_{2^{t-1}},\quad
         \RP^{2^{t}-j}\longrightarrow \RP^{2^{t}-j}_{2^{t-1}},
         \]
         induce the  homology surjections:
         \begin{eqnarray*}
         H_*(Q(\RP^{2^{t}});\FF_2)&\longrightarrow& H_*(Q(\RP^{2^{t}}_{2^{t-1}+1});\FF_2),\\
         H_*(Q(\RP^{2^{t}-1};\FF_2)& \longrightarrow& H_*(Q(\RP^{2^{t}-1}_{2^{t-1}};\FF_2),\\
         H_*(Q(\RP^{2^{t}-j});\FF_2)&\longrightarrow& H_*(Q(\RP^{2^{t}-j}_{2^{t-1}});\FF_2).
         \end{eqnarray*}
         Consequently all differentials arising from the zero row must vanish and the spectral sequence collapses 
         at the $E_2$-term.
\end{compactitem}
Thus the cohomology Serre spectral sequences of the 
fibrations~\eqref{eq:fibration-01},~\eqref{eq:fibration-02} and~\eqref{eq:fibration-03} 
yield the following short exact sequences of Hopf algebras:
\begin{multline}\label{eq:hopf-01}
1\longrightarrow 
H^*(Q(\RP^{2^{t}}_{2^{t-1}+1});\FF_2)\longrightarrow
H^*(Q(\RP^{2^{t}});\FF_2)\\
\longrightarrow H^*(Q(\RP^{2^{t-1}});\FF_2)\longrightarrow 1,
\end{multline}
\begin{multline}\label{eq:hopf-02}
1\longrightarrow 
H^*(Q(\RP^{2^{t}-1}_{2^{t-1}});\FF_2)\longrightarrow 
H^*(Q(\RP^{2^{t}-1});\FF_2)\\
\longrightarrow H^*(Q(\RP^{2^{t-1}-1});\FF_2)
\longrightarrow  1,
\end{multline}
\begin{multline}\label{eq:hopf-03}
1\longrightarrow 
H^*(Q(\RP^{2^{t}-j}_{2^{t-1}});\FF_2)\longrightarrow 
 H^*(Q(\RP^{2^{t}-j});\FF_2)\\
 \longrightarrow H^*(Q(\RP^{2^{t-1}-1});\FF_2)\longrightarrow 1.
\end{multline}
Now we apply  Lemma~\ref{lemma:tools} on sequences~\eqref{eq:hopf-01},~\eqref{eq:hopf-02} and~\eqref{eq:hopf-03} to finish the proof.

\begin{claim}\label{claim:1}
Let $1\leq j\leq 2^{t-1}-1$. Then
\begin{compactenum}[\rm (1)]
\item $\h(H^*(Q(\RP^{2^{t}}_{2^{t-1}+1});\FF_2))=2$,
\item $\h(H^*(Q(\RP^{2^{t}-1}_{2^{t-1}});\FF_2))=2$,
\item $\h(H^*(Q(\RP^{2^{t}-j}_{2^{t-1}});\FF_2))=2$.
\end{compactenum}
\end{claim}
\begin{proof}
The homologies with coefficients in $\FF_2$ of the spaces $\RP^{2^{t}}_{2^{t-1}+1}$,  $\RP^{2^{t}-1}_{2^{t-1}}$ and $\RP^{2^{t}-j}_{2^{t-1}}$ are primitively generated.
Therefore, the cohomologies 
\[
H^*(Q(\RP^{2^{t}}_{2^{t-1}+1});\FF_2),
\qquad
H^*(Q(\RP^{2^{t}-1}_{2^{t-1}});\FF_2),
\qquad
H^*(Q(\RP^{2^{t}-j}_{2^{t-1}});\FF_2),
\]
have height $2$.
\end{proof}

\begin{claim}\label{claim:2}
$\h(H^*(Q(\RP^{2^{t}});\FF_2)\leq 2^{t+1}$.
\end{claim}
\begin{proof}
The proof is by induction on $t$.
For $t=0$ we that $\RP^1$ is primitively generated and consequently  
\[
\h(H^*(Q(\RP^1);\FF_2))=2.
\]
Let us assume that 
\[
\h(H^*(Q(\RP^{2^{t-1}});\FF_2))\leq 2^t,
\]
and consider the short exact sequence of Hopf algebras~\eqref{eq:hopf-01}:
\[
1\longrightarrow
H^*(Q(\RP^{2^{t-1}});\FF_2)\longrightarrow 
H^*(Q(\RP^{2^{t}});\FF_2)\longrightarrow
H^*(Q(\RP^{2^{t}}_{2^{t-1}+1});\FF_2)\longrightarrow 1.
\]
Then from Claim~\ref{claim:1} and induction hypothesis after applying  Lemma~\ref{lemma:tools}  we get that
\[
\h(H^*(Q(\RP^{2^t});\FF_2))\leq 2^{t+1}.
\]
\end{proof}

\begin{claim}\label{claim:3}
$\h(H^*(Q(\RP^{2^{t}-1});\FF_2)\leq 2^{t}$.
\end{claim}
\begin{proof}
The proof is again by induction on $t$.
For $t=1$, as before, we have that $\RP^1$ is primitively generated and consequently $\h(H^*(Q(\RP^1);\FF_2))=2$.
Let us assume that 
\[
\h(H^*(Q(\RP^{2^{t-1}-1}))\leq 2^{t-1}
\]
and consider the exact sequence of Hopf algebras~\eqref{eq:hopf-02}:
\[
1\longrightarrow 
H^*(Q(\RP^{2^{t-1}-1});\FF_2)\longrightarrow 
H^*(Q(\RP^{2^{t}-1});\FF_2)\longrightarrow  H^*(Q(\RP^{2^{t}-1}_{2^{t-1}});\FF_2)\longrightarrow 1.
\]
Now the induction hypothesis and Claim~\ref{claim:1} combined with Lemma~\ref{lemma:tools} imply that
\[
\h(H^*(Q(\RP^{2^{t}-1}))\leq 2^{t-1}\cdot 2.
\]
\end{proof}

\begin{claim}\label{claim:4}
Let $1\leq j\leq 2^{t-1}-1$, then $\h(H^*(Q(\RP^{2^{t}-j});\FF_2)\leq 2^{t}$.
\end{claim}
\begin{proof}
The exact sequence of Hopf algebra~\eqref{eq:hopf-03}:
\[
1\longrightarrow 
H^*(Q(\RP^{2^{t-1}-1});\FF_2)\longrightarrow 
H^*(Q(\RP^{2^{t}-j});\FF_2)\longrightarrow
 H^*(Q(\RP^{2^{t}-j}_{2^{t-1}});\FF_2)\longrightarrow 1
\]
combined with Claims~\ref{claim:1} and~\ref{claim:3}, and Lemma~\ref{lemma:tools} implies that
\[
\h(H^*(Q(\RP^{2^{t}-j});\FF_2)
\leq
2^{t-1}\cdot 2.
\] 
\end{proof}

\noindent
Combining all the claims together we get that
\[
\h(H^*(F(\RR^d,k)/\Sym_k;\FF_2))\leq\h(H^*(Q(\RP^{d-1});\FF_2)\leq \min\{ 2^t: 2^t\geq d \}.
\]
This concludes the proof of Theorem~\ref{theorem:hight_bound_for_prime_2}.

% -------------------------------------------------------------------%
% -------------------------------------------------------------------%
\subsection{Proof of Theorem~\ref{theorem:hight_bound_for_odd}}
% -------------------------------------------------------------------%
% -------------------------------------------------------------------%
The proof is separated into two parts depending on the parity of $d$.

\subsubsection{Proof of Theorem~\ref{theorem:hight_bound_for_odd} for $d$ odd}
As in the proof of Theorem~\ref{theorem:hight_bound_for_prime_2}, 
instead of estimating the heights of the elements in $H^*(F(\RR^d,k)/\Sym_k;\FF_p)$ 
we use Corollary~\ref{corollary:tools-2} that yields inequality
\[
\h(H^*(F(\RR^d,k)/\Sym_k;\FF_p))\leq \h(H^*(Q(F(\RR^d,p)/\Sym_p);\FF_p))
\]
and estimate the heights of elements in $H^*(Q(F(\RR^d,p)/\Sym_p);\FF_p)$ instead. 
The Borel structural theorem~\cite[Prop.\,7.8 and Thm.\,7.11]{milnor - moore} implies that 
\[
H^*(Q(F(\RR^d,p)/\Sym_p);\FF_p)\cong\bigotimes_{i\in I}\FF_p[x_i]/\langle {x_i}^{s_i}\rangle\otimes\bigotimes_{j\in J}\Lambda(y_j)
\]
where each height $s_i$ is either power of $p$ or $\infty$ and $\deg(y_j)$ is odd for every $j\in J$. 
Estimation of the height of  algebra $H^*(Q(F(\RR^d,p)/\Sym_p);\FF_p)$ will be done in several steps.

First we recall the following special cases, Lemma~\ref{lemma:hight of Q(cofiber)}.

\begin{claim}\label{claim:mod-p-01}
For  $3\leq d\leq 2p-1$ an odd integer:
\[
\h(H^*(Q(F(\RR^{d},p)/\Sym_p);\FF_p))\leq p.
\]
\end{claim}

Now we use the previous claim as an induction step in the next claim.

\begin{claim}\label{claim:mod-p-02}
For $j\geq0$ an integer and $2p^j+1\leq d\leq 2p^{j+1}-1$ be an odd integer:
\[
\h(H^*(Q(F(\RR^{d},p)/\Sym_p);\FF_p))\leq p^{j+1}.
\]
\end{claim}
\begin{proof}
For simplicity when $j\geq0$ we denote by:
\[
a_j:=2p^j-1
\quad<\quad
b_{j,k}:=2p^{j+1}+1-k,
\]
where $2\leq k \leq 2p^j(p-1)$ is an even integer. 
Considering the cofibration
\begin{equation}\label{eq:p-cofibration-01}
 \xymatrix@!
{
F(\RR^{a_j},p)/\Sym_p\ar[r]& F(\RR^{b_{j,k}},p)/\Sym_p\ar[r]& C^{b_{j,k}}_{a_j}.
}
\end{equation}
Now we apply $Q(\cdot)$ and get, from Lemma~\ref{lemma:cofib--fib--Q(X)}, the
following fibration:
\begin{equation}\label{eq:p-fibration-01}
 \xymatrix@!
{
Q(F(\RR^{a_j},p)/\Sym_p)\ar[r]& Q(F(\RR^{b_{j,k}},p)/\Sym_p)\ar[r]& Q(C^{b_{j,k}}_{a_j}).
}
\end{equation}
Both the homology and the cohomology Serre spectral sequence of the fibration
\eqref{eq:p-fibration-01} collapse at the second term, because:
\begin{compactitem}[ $\bullet$]
\item $Q(C^{b_{j,k}}_{a_j})$ is simply connected and so the second term is a given by the
  tensor product of the (co)homology of the fiber and base (co)homology;

\item the fibration~\eqref{eq:p-fibration-01} is multiplicative, and therefore the
  Leibniz rule for the differentials holds,~\cite{felix-halperin-thomas}.  Therefore, the
  spectral sequence is completely determined by the differential arising from the zero
  row, and finally

\item the quotient map
$
F(\RR^{b_{j,k}},p)/\Sym_p\longrightarrow C^{b_{j,k}}_{a_j}
$
induces a surjection in homology
\[
H_*(F(\RR^{b_{j,k}},p)/\Sym_p;\FF_p)\longrightarrow H_*(C^{b_{j,k}}_{a_j};\FF_p)
\]
that further on induces the following surjection in homology
\[
H_*(Q(F(\RR^{2p^{j+1}+1},p)/\Sym_p);\FF_p)\longrightarrow H_*(Q(C^{2p^{j+1}+1}_{2p^j+1});\FF_p).
\]
Thus all differentials starting in the zero row have to vanish.
\end{compactitem}
Consequently, the cohomology Serre spectral sequence of the
fibration~\eqref{eq:p-fibration-01} gives the following exact sequence of Hopf
algebras:
\begin{multline}\label{eq:Hopf-fibration-01}
1\longrightarrow
H^*(Q(C^{b_{j,k}}_{a_j});\FF_p)\longrightarrow H^*(Q(F(\RR^{b_{j,k}},p)/\Sym_p);\FF_p)\\
\longrightarrow H^*(Q(F(\RR^{a_j},p)/\Sym_p);\FF_p)
\longrightarrow 1.
\end{multline}
Since 
\[
p(a_j+1)-b_{j,k}=\cancel{2p^{j+1}}-\cancel{2p^{j+1}}-1+k=k-1>0,
\] 
then Lemma~\ref{lemma:hight of Q(cofiber)-2}, implies that for every $j\geq0$:
\begin{equation}\label{eq:height-01}
\h(H^*(Q(C^{b_{j,k}}_{a_j});\FF_p))\leq p.
\end{equation}
According to Claim~\ref{claim:mod-p-01} we have that:
\begin{equation}\label{eq:height-01-1}
\h(H^*(Q(F(\RR^{2p-1},p)/\Sym_p);\FF_p))\leq p.
\end{equation}
Consequently, from inequalities~\eqref{eq:height-01} with $j=1$ and~\eqref{eq:height-01-1}, 
exact sequence of Hopf algebras~\eqref{eq:Hopf-fibration-01} with $j=1$, and Lemma~\ref{lemma:tools} we get that
\[
\h(H^*(Q(F(\RR^{b_{1,k}},p)/\Sym_p);\FF_p))\leq p^{2}.
\]
for every $2\leq k\leq 2p(p-1)$.

\noindent
Since $a_{j+1}=b_{j,2}$ we can iterate the process.
Indeed, we have just proved that 
\[\h(H^*(Q(F(\RR^{a_2},p)/\Sym_p);\FF_p))\leq p^{2}.\] 
Then from inequality~\eqref{eq:height-01} with $j=2$,  exact sequence of 
Hopf algebras~\eqref{eq:Hopf-fibration-01} with $j=2$, and Lemma~\ref{lemma:tools} we conclude that:
\[
\h(H^*(Q(F(\RR^{b_{2,k}},p)/\Sym_p);\FF_p))\leq p^{3},
\]
for all relevant $k$.

\noindent
The proof is concluded by induction on $j$ by assuming
\[
\h(H^*(Q(F(\RR^{a_j},p)/\Sym_p);\FF_p))\leq p^j
\]
and then evoking inequality~\eqref{eq:height-01}, exact sequence of 
Hopf algebras~\eqref{eq:Hopf-fibration-01}, and Lemma~\ref{lemma:tools} to get:
\[
\h(H^*(Q(F(\RR^{b_{j,k}},p)/\Sym_p);\FF_p))\leq p^{j+1},
\]
for relevant $k$.
\end{proof}

The general case is obtained along the lines of the previous claim as follows.
Combining the results of Claims~\ref{claim:mod-p-01} and~\ref{claim:mod-p-02} with the observation
\[
\h(H^*(F(\RR^d,k)/\Sym_k;\FF_p))\leq \h(H^*(Q(F(\RR^d,p)/\Sym_p);\FF_p))
\]
we get that for $d\leq 2p^{j}-1$ where $j\geq 1$ holds
\[
 \h(H^*(F(\RR^d,k)/\Sym_k;\FF_p))\leq p^{j}=\min\{ p^t: 2p^t\geq d \}.
\]
This concludes the proof of Theorem~\ref{theorem:hight_bound_for_odd}.

% -------------------------------------------------------------------%
% -------------------------------------------------------------------%
\subsubsection{Proof of Theorem~\ref{theorem:hight_bound_for_odd} for $d$ even}
% -------------------------------------------------------------------%
% -------------------------------------------------------------------%

Let $d\geq 2$ be an even integer. 
The well known result of Serre asserts that the homotopy group $\pi_{2d-1}(S^{d})$ is a direct sum of the infinite cyclic group and a finite group.
Let $f\colon S^{2d-1}\rightarrow S^d$ represents a generator of a infinite cyclic summand. 
Further on, let $g\colon S^{d-1}\rightarrow\Omega S^d$ be the map corresponding to the identity map $S^d\rightarrow S^d$. 
The Serre decomposition \cite{serre}  is defined by the following map 
\[
\xymatrix{
S^{d-1}\times \Omega S^{2d-1}\ar[r]^-{g\times\Omega(f)}&
\Omega S^d\times\Omega S^d\ar[r]^-{\times}&
\Omega S^d
}
\]
that is homotopy equivalence when all the spaces are localized away from $2$, \cite[Prop.\,4.4.4]{neisendorfer}.

\smallskip
\noindent
The localization commutes with taking loops.
Thus, after looping $d-1$ times we get a homotopy equivalence
\[
\Omega^{d-1}S^{d-1}\times \Omega^d S^{2d-1}\xrightarrow{\Omega^{d-1}(g)\times\Omega^d(f)}
\Omega^d S^d\times\Omega^d S^d\xrightarrow{\quad \times \quad }
\Omega^d S^d
\]
when all the spaces are localized away from $2$.
Since $\Omega^d(f)$ factors as follows:
\[
\xymatrix{
\Omega^dS^{2d-1}\ar[rr]^-{\Omega^d(f)}\ar[dr]&   & \Omega^dS^d\\
&\Omega^d_0S^d\ar@{^{(}->}[ur],&
}
\]
we have a homotopy equivalence of corresponding connected components
\[
\xymatrix{
\varphi\colon\Omega^{d-1}_0S^{d-1}\times \Omega^d S^{2d-1}\ar[r]&
\Omega^d_0 S^d,
}
\]
with all the spaces localized away from $2$.

\smallskip
\noindent
Consider the following diagram
\[
\xymatrix@!C=16em{
F(\RR^d,k)/\Sym_k\ar[r]^-{\gamma_{d,k}}&
\Omega^d_0 S^d\\
Q(F(\RR^{d-1},p)/\Sym_p)\times \Omega^dS^{2d-1}&                                                 
\ar[l]_-{t_{d,p}\times\id}\Omega^{d-1}_0S^{d-1}\times \Omega^d S^{2d-1}\ar[u]^{\varphi}
}
\]
where in cohomology with coefficients in $\FF_p$: 
\begin{compactitem}
\item $\gamma_{d,k}$ induces an epimorphism, by Corollary~\ref{corollary:tools-1},
\item $\varphi$ induces an isomorphism, as we just established, and
\item $t_{d,p}\times\id$ induces also an epimorphism, by Corollary~\ref{corollary:tools-3}.
\end{compactitem}
Thus, there is a sequence of estimates:
\begin{eqnarray*}
\h(H^*(F(\RR^d,k)/\Sym_k;\FF_p))&\leq& \h(H^*(\Omega^d_0S^d;\FF_p))\\
&=&\h(H^*(\Omega^{d-1}_0S^{d-1}\times \Omega^d S^{2d-1};\FF_p))\\
&\leq&\max\{\h(H^*(\Omega^{d-1}_0S^{d-1};\FF_p)),\\
& & \,\,\qquad\h(H^*(\Omega^d S^{2d-1};\FF_p))\}\\
&\leq&\max\{\h(H^*(Q(F(\RR^{d-1},p)/\Sym_p);\FF_p)),\\
& & \,\,\qquad\h(H^*(\Omega^d S^{2d-1};\FF_p))\}.
\end{eqnarray*}
Finally, since $H_*(\Omega^dS^{2d-1};\FF_p)$ is primitively 
generated,~\cite[Proof of Thm.\,5.2]{dyer-lahof}, the height of its cohomology is bounded by $p$, i.e.,
$
\h(H^*(\Omega^d S^{2d-1};\FF_p))\leq p
$.
Consequently,
\[
\h(H^*(F(\RR^d,k)/\Sym_k);\FF_p)\leq \h(H^*(Q(F(\RR^{d-1},p)/\Sym_p);\FF_p)),
\]
and this concludes the proof of Theorem~\ref{theorem:hight_bound_for_odd} in case $d$ is even.

%%%%%%%%%%%%%%%%%%%%%%%%%%%%%%%%%%%%%%%%%%%%%%%%%%%%%%%%%%%%%%%%%%%%%%%%%%%%%%%%%%%%%
%%%%%%%%%%%%%%%%%%%%%%%%%%%%%%%%%%%%%%%%%%%%%%%%%%%%%%%%%%%%%%%%%%%%%%%%%%%%%%%%%%%%%
\section{Chern classes of the regular representation bundle}
%%%%%%%%%%%%%%%%%%%%%%%%%%%%%%%%%%%%%%%%%%%%%%%%%%%%%%%%%%%%%%%%%%%%%%%%%%%%%%%%%%%%%
%%%%%%%%%%%%%%%%%%%%%%%%%%%%%%%%%%%%%%%%%%%%%%%%%%%%%%%%%%%%%%%%%%%%%%%%%%%%%%%%%%%%%

Let $\FF$ denote the field of either the real numbers $\RR$ or the complex numbers~$\CC$.
Consider  $\FF^k$ as an $\Sym_k$-representation with the action given by the coordinate permutation.  
Then the vector subspace $W_k^{\FF}=\{(x_1,\ldots,
x_k)\in\FF^k:\sum x_i=0\}$ of $\FF^k$ is an $\Sym_k$-subrepresentation.  
For a topological space $X$ we introduce the following vector bundles over its unordered configuration space
\[
 \xymatrix{
 \xi_{X,k}^{\FF}      & \FF^k\ar[r]  &  F(X,k)\times_{\Sym_k}\FF^k\ar[r]& F(X,k)/\Sym_k,}
 \]
 \[
\xymatrix{
\zeta_{X,k}^{\FF} & W_k^{\FF} \ar[r]  & F(X,k)\times_{\Sym_k}W_k^{\FF}\ar[r]&    F(X,k)/\Sym_k,}
\]
\[
\xymatrix{
\tau_{X,k}^{\FF}  & \FF\ar[r]    & F(X,k)/\Sym_k\times\FF \ar[r]& F(X,k)/\Sym_k,}
\]
where the last bundle is a trivial $\FF$ line bundle.
The bundle $\xi_{X,k}^{\FF}$ is called {\em the regular representation bundle} over the field $\FF$.
There is an obvious decomposition:
\begin{equation}
\label{eq:decomposition}
\xi_{X,k}^{\FF}\cong\zeta_{X,k}^{\FF}\oplus \tau_{X,k}^{\FF}.
\end{equation}

\smallskip
Let $\alpha$ be an $n$-dimensional complex vector bundle over $X$. 
We denote by 
\[
c_i(\alpha)\in H^{2i}(X;\FF_p)
\]
the $i$th Chern classes of $\alpha$ modulo prime $p$, $1\leq i\leq n$.
The total Chern class is denoted by:
\[
c(\alpha)=1+c_1(\alpha)+\cdots+c_n(\alpha)\in H^*(X;\FF_p).
\]
Define the \emph{inverse Chern class} modulo $p$ to be
\[
\overline{c}(\alpha)=c(\alpha)^{-1}=\sum_{n\geq 0}(-1)^n(c_1(\alpha)+\cdots+c_n(\alpha))^n.
\]
If $\beta$ is a complex vector bundle over $X$ such that the Whitney sum 
$\alpha\oplus\beta$ is isomorphic to a trivial bundle, then
\[
c(\beta)=\overline{c}(\alpha).
\]

\smallskip
In this section, combining the methods from papers \cite{blagojevic-ziegler-convex} and \cite{blagojevic-luck-ziegler-1} we prove the following theorems.
\begin{theorem}
\label{theorem : Chern classes}
Let $d\geq 1$ be an integer and $k$ be a power of an odd prime $p$.
Then the mod $p$ Chern class
\begin{equation} 
c_{(d-1)(k-1)}((\xi^{\CC}_{\CC^d,k})^{\oplus d-1})=c_{k-1}(\xi^{\CC}_{\CC^d,k})^{d-1}\neq 0\
\end{equation}
does not vanish in $H^{2(d-1)(k-1)}(F(\CC^d,k)/\Sym_k;\FF_p)$.
\end{theorem}

\noindent
Observe that, from decomposition \eqref{eq:decomposition} and dimensional reasons,  the following Chern classes vanish
\[
c_{m}((\xi^{\CC}_{\CC^d,k})^{\oplus d-1})=c_{m}((\zeta^{\CC}_{\CC^d,k})^{\oplus d-1})=0
\]
when $m\geq (d-1)(k-1)+1$, and $d\geq1$, $k\geq 2$ arbitrary.

Let $p$ be a prime.
Then every integer $k\geq 1$ can be presented uniquely in the form
\[
k=\beta_1p^{r_1}+\cdots+\beta_ap^{r_a}
\]
where $0\leq r_1<r_2<\cdots<r_a$, and $0<\beta_i<p$ for all $1\leq i\leq a$.
We define the function $\alpha_p\colon\NN\rightarrow\NN$ by:
\[
\alpha_p(k):=\beta_1+\cdots+\beta_a.
\]

\begin{theorem}
\label{theorem : Chern classes-2}
Let $d\geq 1$ and $k\geq 2$ be integers and $p$ be an odd prime.
Then the mod $p$ Chern class
\begin{equation} 
c_{(d-1)(k-\alpha_p(k))}((\xi^{\CC}_{\CC^d,k})^{\oplus d-1})\neq 0\
\end{equation}
does not vanish in $H^{2(d-1)(k-\alpha_p(k))}(F(\CC^d,k)/\Sym_k;\FF_p)$.
\end{theorem}

The previous theorems, as a side results, imply new estimates of the Lusternik--Schnirelmann category of the unordered configuration space $F(\CC^d,k)/\Sym_k$ as well as of the sectional category of the covering $F(\CC^d,k)\longrightarrow F(\CC^d,k)/\Sym_k$.

The \emph{Lusternik--Schnirelmann category} of a space $X$, denoted by $\cat(X)$, is the least
integer $n$ for which there exists an open cover $U_1, U_2,
\ldots, U_{n+1}$ of $X$ such that the inclusions $U_i \to X$ are nullhomotopic.

\noindent 
The \emph{sectional category} $\secat(p)=\secat(E\overset{p}{\to}B)$ of the fibration $F\to E\overset{p}{\to}B$ is the minimal integer $n$ for which $B$ can be covered by $n+1$ open subsets $U_1, U_2,\ldots, U_{n+1}$ with the property that each restriction fibration $F\to p^{-1}(U_i)\to U_i$ admits a section $s_i\colon U_i\to p^{-1}(U_i)$.
Originally, the notion of sectional category was introduced by Schwarz in \cite{schwarz} under the name {\em genus}. 

\smallskip
A direct consequence of Theorem~\ref{theorem : Chern classes-2} is the following estimate of the Lusternik–-Schnirelmann category of the configuration space $F(\CC^d,k)/\Sym_k$.

\begin{corollary}
\label{cor:LS-category}
$\cat (F(\CC^d,k)/\Sym_k)\geq \max\{2(d-1)(k-\alpha_p(k)): p\text{ is a prime}\}$.
\end{corollary}

\noindent
This corollary, in some cases, improves the lower bound given by Roth \cite[Thm.\,1.4]{Roth}:
\[
\cat (F(\CC^d,k)/\Sym_k)\geq (2d-1)(k-\alpha_2(k)).
\]
Indeed, for example in the case $k=14$, we have that
\[
(2d-1)(k-\alpha_2(k))=11(2d-1)=22d-11,
\]
while
\[
(2d-2)(k-\alpha_7(k))=12(2d-2)=24d-24.
\]
Note that in this case the bound from the corollary improves the bound given by Roth only for $d$ large enough.

\smallskip
The following new bound for the sectional category of the covering $F(\CC^d,k)\longrightarrow F(\CC^d,k)/\Sym_k$, in the case when $k$ is a power of an odd prime, is a consequence of Theorem~\ref{theorem : Chern classes} combined with \cite[Prop. 2.1(3) and 2.2(3)]{Roth} and \cite[Thm. 8.4]{blagojevic-luck-ziegler-1}.

\begin{corollary}
\label{cor:secat-category}
Let $d\geq 1$ be an integer and $k$ be a power of an odd prime $p$. 
Then
\begin{multline*}
(2d-2)(k-1)\leq\secat\big(F(\CC^d,k)\longrightarrow F(\CC^d,k)/\Sym_k\big)\\
\leq\cat(F(\CC^d,k)/\Sym_k)=(2d-1)(k-1).
\end{multline*}
\end{corollary}

The problem of estimating  $\secat\big(F(\CC^d,k)\longrightarrow F(\CC^d,k)/\Sym_k\big)$ was extensively studied by Vassiliev~\cite{vassiliev96}, De Concini \& Procesi \& Salvetti~\cite{deConcini-procesi-salvetti}, Arone~\cite{arone2006} and Roth~\cite{Roth}.

% -------------------------------------------------------------------%
% -------------------------------------------------------------------%
\subsection{Proof of Theorem~\ref{theorem : Chern classes}}
% -------------------------------------------------------------------%
% -------------------------------------------------------------------%
Let $\mathcal{Z}:=\ZZ$ and $\mathcal{F}_p:=\FF_p$ be $\Sym_k$-modules with action defined by $\pi\cdot z=(-1)^{\mathrm{sgn}(\pi)}z$ for $\pi\in\Sym_k$ and $z\in \mathcal{Z}$, or $z\in \mathcal{F}_p$. 
According to \cite[Thm.\,1.2]{blagojevic-ziegler-convex} and \cite[Lemma 5.2]{blagojevic-luck-ziegler-1} we know that the Euler class of the vector bundle $(\zeta_{\RR^{2d},k}^{\RR})^{\oplus 2d-1}$ with twisted $\mathcal{Z}$ coefficients
\[
\mathfrak{e}(\zeta_{\RR^{2d},k}^{\RR};\mathcal{Z})^{2d-1}\in H^*(F(\RR^{2d},k)/\Sym_k;\mathcal{Z})=H^*(F(\CC^{d},k)/\Sym_k;\mathcal{Z})
\]
does not vanish.
For details on Euler classes  with twisted coefficients consult for example \cite{greenblatt}.
Further on,  \cite[Lemma 4.2]{blagojevic-ziegler-convex} guarantees that after reduction of coefficients $\mathcal{Z}\rightarrow \mathcal{F}_p$, $z\mapsto z\textrm{ mod }p$, the reduced Euler class 
\[
\mathfrak{e}(\zeta_{\RR^{2d},k}^{\RR};\mathcal{F}_p)^{2d-1}\in H^*(F(\RR^{2d},k)/\Sym_k;\mathcal{F}_p)=H^*(F(\CC^{d},k)/\Sym_k;\mathcal{F}_p)
\]
also does not vanish.

\noindent 
There is an isomorphism of real vector bundles
$2\zeta_{\RR^{2d},k}^{\RR}=\zeta_{\CC^{d},k}^{\CC}$.
The multiplicative property of twisted Euler classes \cite[Thm.\,3.3]{greenblatt} implies that
\[
0 \neq \mathfrak{e}(\zeta_{\RR^{2d},k}^{\RR};\mathcal{F}_p)^{2d-1}
     =     \mathfrak{e}(\zeta_{\RR^{2d},k}^{\RR};\FF_p)^{2d-2}\cdot \mathfrak{e}(\zeta_{\RR^{2d},k}^{\RR};\mathcal{F}_p)\\    
     =     \mathfrak{e}(\zeta_{\CC^{d},k}^{\CC};\FF_p)^{d-1}\cdot \mathfrak{e}(\zeta_{\RR^{2d},k}^{\RR};\mathcal{F}_p).
\]
Consequently, 
\[
\mathfrak{e}(\zeta_{\CC^{d},k}^{\CC};\FF_p)^{d-1}\neq 0.
\]
Now, using the isomorphism 
$\xi_{\CC^d,k}^{\CC}\cong\zeta_{\CC^d,k}^{\CC}\oplus \tau_{\CC^d,k}^{\CC}$
and definition of the top Chern class~\cite[Def., pp.\,158]{milnor-stasheff}, we get that
\[
c_{k-1}(\xi_{\CC^{d},k}^{\CC})^{d-1}=
c_{k-1}(\zeta_{\CC^{d},k}^{\CC})^{d-1}=\mathfrak{e}(\zeta_{\CC^{d},k}^{\CC};\FF_p)^{d-1}\neq 0.
\]
This concludes the proof.

% -------------------------------------------------------------------%
% -------------------------------------------------------------------%
\subsection{Proof of Theorem~\ref{theorem : Chern classes-2}}
% -------------------------------------------------------------------%
% -------------------------------------------------------------------%
Let $k=\beta_1p^{r_1}+\cdots+\beta_ap^{r_a}$ where $0\leq r_1<r_2<\cdots<r_a$, and $0<\beta_i<p$ for all $1\leq i\leq a$.
We construct a morphism of fiber bundles $\prod_{t=1}^{a}\prod_{u=1}^{\beta_t}\xi_{\CC^d,p^{r_t}}$ and $\xi_{\CC^d,k}$
  such that the following commutative square is a pullback diagram: 
 \[
\xymatrix{\prod_{t=1}^{a}\prod_{u=1}^{\beta_t}\xi_{\CC^d,p^{r_t}}\ar[r]^-{\Theta}  \ar[d] 
& \xi_{\CC^d,k}   \ar[d] 
\\
\prod_{t=1}^{a}\prod_{u=1}^{\beta_t}F(\CC^d,p^{r_t})/\Sym_{p^{r_t}}  \ar[r]_-{\theta}   & 
F(\CC^d,k)/{\Sym_k}.
}
\]
Choose embeddings $e_i \colon \CC^d \longrightarrow \CC^d$ for $i = 1,2, \ldots , \alpha_p(k)$ such that their images are pairwise disjoint open $2d$-balls. 
Each embedding $e_i$ induces an embedding on the level of configuration spaces $F(\CC^d,\ell) \longrightarrow F(\CC^d,\ell)$ denoted by the same letter $e_i$ for any natural numbers $\ell$. 
The map $\theta\colon \prod_{t=1}^{a}\prod_{u=1}^{\beta_t}F(\CC^d,p^{r_t})/\Sym_{p^{r_t}} \longrightarrow F(\CC^d,k)$ is given by
\[
(\underline{x_1} ,\ldots , \underline{x_{\alpha_p(k)}})  \mapsto e_1(\underline{x_1}) \times \cdots \times e_{\alpha_p(k)}(\underline{x_{\alpha_p(k)}}).
\]
The map of the total spaces $\Theta$ is then given by
\[
\bigl((\underline{x_1},v_1), \ldots ,(\underline{x_{\alpha_p(k)}},v_{\alpha_p(k)})\bigr)
\mapsto
\bigl(e_1(\underline{x_1}) \times \cdots \times e_{\alpha_p(k)}(\underline{x_{\alpha_p(k)}}), v_1\times \cdots\times  v_{\alpha_p(k)}\bigr).
\]
Therefore, the pullback bundle is indeed a direct product bundle
\[
\theta^*\xi_{\CC^d,k} \cong  \prod_{t=1}^{a}\prod_{u=1}^{\beta_t}\xi_{\CC^d,p^{r_t}}.
\]

\noindent
The naturality of the Chern classes \cite[Lemma 14.2]{milnor-stasheff} combined with the Product theorem \cite[(14.7)]{milnor-stasheff} yields  
\[
\theta^*c(\xi_{\CC^d,k}^{\oplus d-1})=
c\Big(   \prod_{t=1}^{a}\prod_{u=1}^{\beta_t}\xi_{\CC^d,p^{r_t}}^{\oplus d-1}  \Big)=
 \bigtimes_{t=1}^{a}\bigtimes_{u=1}^{\beta_t}c(\xi_{\CC^d,p^{r_t}}^{\oplus d-1})
\]
Therefore,
\[
 \theta^*c_{(d-1)(k-\alpha_p(k))}(\xi_{\CC^d,k}^{\oplus d-1})
 =
  \bigtimes_{t=1}^{a}\bigtimes_{u=1}^{\beta_t}c_{(d-1)(p^{r_t}-1)}(\xi_{\CC^d,p^{r_t}}^{\oplus d-1}).
\]
The K\"unneth formula and Theorem \ref{theorem : Chern classes} imply that
\[
\theta^*c_{(d-1)(k-\alpha_p(k))}(\xi_{\CC^d,k}^{\oplus d-1})\neq 0
\]
and consequently $c_{(d-1)(k-\alpha_p(k))}(\xi_{\CC^d,k}^{\oplus d-1})$ does not vanish.
%%%%%%%%%%%%%%%%%%%%%%%%%%%%%%%%%%%%%%%%%%%%%%%%%%%%%%%%%%%%%%%%%%%%%%%%%%%%%%%%%%%%%
%%%%%%%%%%%%%%%%%%%%%%%%%%%%%%%%%%%%%%%%%%%%%%%%%%%%%%%%%%%%%%%%%%%%%%%%%%%%%%%%%%%%%
\section{$k$-regular embeddings over $\CC$}
%%%%%%%%%%%%%%%%%%%%%%%%%%%%%%%%%%%%%%%%%%%%%%%%%%%%%%%%%%%%%%%%%%%%%%%%%%%%%%%%%%%%%
%%%%%%%%%%%%%%%%%%%%%%%%%%%%%%%%%%%%%%%%%%%%%%%%%%%%%%%%%%%%%%%%%%%%%%%%%%%%%%%%%%%%%
In this section we introduce the notion of a continuous $k$-regular embedding over $\CC$ and prove the following theorems.

\begin{theorem}\label{theorem_0}
Let $d\geq 1$ be an integer. 
There is no complex $k$-regular embedding $\RR^d\longrightarrow\CC^N$ provided that
\[
N< \tfrac{1}{2}(d(k-\alpha(k))+\alpha(k)).
\]
\end{theorem}

\noindent
This is a consequence of our result on real $k$-regular maps~\cite{blagojevic-luck-ziegler-2}.

\begin{theorem} \label{theorem_1}
Let $d\geq 1$ be an integer, and $p$ be an odd prime.
There is no complex $p$-regular embedding $\RR^d\longrightarrow\CC^N$ provided that
\[
N\leq \lfloor \tfrac {d+1}{2}\rfloor (p-1).
\]
\end{theorem}

\noindent
The theorem is tight in the case $d=1,2$ by Example~\ref{example_1}.

\begin{theorem} \label{theorem_2}
Let $p$ be an odd prime, $k\geq 1$ and $d=p^t$  for $t\geq1$.
There is no complex $k$-regular embedding $\CC^d\longrightarrow\CC^N$ provided that
\[
N\leq d(k-\alpha_p(k))+\alpha_p(k)-1.
\]
\end{theorem}

\noindent
The theorem is tight in the case $d=1$ again by Example~\ref{example_1}.
In case $k=p$ is an odd prime and $d$ is power of $p$ Theorems~\ref{theorem_1} and \ref{theorem_2} give the same bound.

\smallskip
The following table compares lower bounds of Theorems \ref{theorem_0}, \ref{theorem_1} and \ref{theorem_2} for the existence of complex $k$-regular embedding $\CC^d\longrightarrow\CC^N$.  
\medskip
\begin{center}
{\small
\begin{tabular}{| l | c | c | c |}
\hline
$(d,k,p)$&$N\geq d(k-\alpha(k))+\lceil\tfrac{\alpha(k)}{2}\rceil$&$N\geq d(p-1)+1$&$N\geq d(k-\alpha_p(k))+\alpha_p(k)$\\    
\hline
$(3,3,3)$&  $4$  &  $7$  &  $7$   \\  
\hline
$(3,9,3)$&  $22$  &  --  &  $25$  \\  
\hline
$(3,8,3)$&  $22$  &  --  &  $18$  \\
\hline
$(d,7,7)$&$4d+2$&$6d+1$& --\\
\hline
$(d,17,17)$&$15d+1$&$16d+1$& --\\
\hline
\end{tabular}
}
\end{center}

\smallskip
The next two subsection are adapted from the~\cite[Sec.\,2.1 and 2.2]{blagojevic-luck-ziegler-2}.

% -------------------------------------------------------------------%
% -------------------------------------------------------------------%
\subsection{Definition and first bound}
% -------------------------------------------------------------------%
% -------------------------------------------------------------------%

\begin{definition}
Let $X$ be a topological space and $k\geq 1$ be an integer.
A continuous map $f\colon X \longrightarrow\CC^N$ is a {\em complex $k$-regular embedding} if for every 
$(x_1,\ldots,x_k)\in F(X,k)$ the set of vectors $\{f(x_1),\ldots,f(x_n)\}$ is 
linearly independent in the complex vector space $\CC^N$.
\end{definition}

\begin{example}
\label{example_1}
{\normalfont
The map $f\colon\CC\longrightarrow\CC^{k}$ given by $f(z)=(1,z,z^2,\ldots,z^{k-1})$ is a complex $k$-regular embedding due to the nonvanishing of the Vandermonde determinant at every point of $F(\CC,k)$.
}
\end{example}

The first necessary condition for the existence of a $2k$-regular embeddings over $\CC$ between complex vector spaces is analogous to the 
Boltjanski{\u\i}, Ry{\v{s}}kov \& {\v{S}}a{\v{s}}kin~\cite{boltjanskii-ryskov-saskin} bound given for the classical real case
and proved by a ``dimension count''.

\begin{theorem}
\label{the:BRS}
If there exists a complex $2k$-regular embedding  $f\colon\RR^d\longrightarrow\CC^N$, then 
\[
\frac{d\cdot k}{2} + k \leq N.
\]
\end{theorem}

% -------------------------------------------------------------------%
% -------------------------------------------------------------------%~
\subsection{A criterion in terms of Chern classes}
% -------------------------------------------------------------------%
% -------------------------------------------------------------------%

Now we derive a criterion for the non-existence of a complex $k$-regular embedding $X\longrightarrow\CC^{N}$
in terms of the mod $p$ Chern classes of the bundle $\xi_{X,k}^{\CC}$ 
where $p$ is an odd prime.

The argument for the real case and Stiefel-Whitney classes appearing in the proof 
of~\cite[Lemma~2.12~(2)]{blagojevic-luck-ziegler-2} carries over to the complex case and 
 Chern classes modulo $p$ and hence we get for the vector bundle
$\xi_{X,k}^{\CC}  \colon F(X,k)\times_{\Sym_k}\CC^k \longrightarrow F(X,k)/\Sym_k$.

\begin{lemma}
\label{lemma:criterion_no_2}
~~
\begin{compactenum}[\em (i)]
\item If there exists a complex $k$-regular embedding  $X\longrightarrow\CC^N$, then the complex vector bundle $\xi_{X,k}^{\CC}$ admits an $(N-k)$-dimensional complex inverse. 
\item If $\overline{c}_{N-k+1}(\xi_{X,k}^{\CC})\neq 0$, then there is no complex $k$-regular embedding  $X\longrightarrow\CC^N$.
\end{compactenum}
\end{lemma}

% -------------------------------------------------------------------%
% -------------------------------------------------------------------%
\subsection{Proof of Theorem~\ref{theorem_0}}
% -------------------------------------------------------------------%
% -------------------------------------------------------------------%
Let  $\RR^d\longrightarrow\CC^N$ be a complex $k$-regular embedding.
According to Lemma~\ref{lemma:criterion_no_2} the complex vector bundle $\xi_{\RR^d,k}^{\CC}$ admits an $(N-k)$-dimensional complex inverse. 
Consequently, the real vector bundle $2\xi_{\RR^d,k}^{\RR}$ admits a $2(N-k)$-dimensional real inverse, and so $\xi_{\RR^d,k}^{\RR}$ admits a $(2N-k)$-dimensional real inverse.
Since, by \cite[Thm. 2.13]{blagojevic-luck-ziegler-2}, $\bar{w}_{(d-1)(k-\alpha(k))}(\xi_{\RR^d,k}^{\RR})\neq 0$ we have that:
\[
2N-k\geq (d-1)(k-\alpha(k))\Longleftrightarrow N\geq \tfrac{1}{2}(d(k-\alpha(k))+\alpha(k)).
\]

% -------------------------------------------------------------------%
% -------------------------------------------------------------------%
\subsection{Proof of Theorem~\ref{theorem_1}}
% -------------------------------------------------------------------%
% -------------------------------------------------------------------%

Let $p$ be a prime and $\ZZ/p$ a subgroup of $\Sym_p$ generated by the cyclic permutation $(123\ldots p)$.
Consider the following bundle maps:
\[
\xymatrix@!C=9em
{
F(\RR^d,p)\times_{\ZZ/p}W_p^{\CC}\ar[r]\ar[d]^{\gamma_{\RR^d,p}^{\CC}}\ar[rdd]|!{[dd];[r]}\hole &
F(\RR^d,p)\times_{\Sym_p}W_p^{\CC}\ar[r]\ar[d]^{\zeta_{\RR^d,p}^{\CC}} &
\EE\Sym_p\times_{\Sym_p}W_p^{\CC}\ar[d]^{\eta_p^{\CC}}\\
F(\RR^d,p)/\ZZ/p\ar[r]^-{r}\ar[rdd]^{a}&
F(\RR^d,p)/\Sym_p\ar[r]^{s}&
\BB\Sym_p\\
&\EE\ZZ/p\times_{\ZZ/p}W_p^{\CC}\ar[d]^{\nu_p^{\CC}}\ar[ruu]|!{[uu];[r]}\hole& \\
&\BB\ZZ/p\ar[ruu]^{b}&
}
\]
Then
\[
\gamma_{\RR^d,p}^{\CC}=r^*\zeta_{\RR^d,p}^{\CC}=(s\circ r)^*\eta_p^{\CC} =(b\circ a)^*\eta_p^{\CC}.
\]
Put
\[M =  \lfloor \tfrac {d+1}{2}\rfloor (p-1)  -p +1 = \lfloor \tfrac {d-1}{2}\rfloor (p-1).
\]
According to Lemma~\ref{lemma:criterion_no_2}, in order to prove the theorem, is suffices to prove
\[
 \overline{c}_M(\xi_{\RR^d,p}^{\CC})\neq 0.
\]
Since  
\[
\overline{c}_M(\xi_{\RR^d,p}^{\CC}) =  \overline{c}_M(\zeta_{\RR^d,p}^{\CC} \oplus \tau^{\CC}_{X,k})
= \overline{c}_M(\zeta_{\RR^d,p}^{\CC})
\] 
and $\overline{c}_M(\gamma_{\RR^d,p}^{\CC})= r^*(\overline{c}_M(r\zeta_{\RR^d,p}^{\CC}))$
it is enough to prove 
\[
\overline{c}_M(\gamma_{\RR^d,p}^{\CC}) \neq 0.
\]
Denote the cohomology of $\BB\ZZ/p$ with coefficients in the field $\FF_p$ by
\[
H^*(\BB\ZZ/p;\FF_p)=\FF_p[t]\otimes \Lambda(e)
\]
where $\deg(e)=1$, $\deg(t)=2$ and $\Lambda(\cdot)$ denotes the exterior algebra.
Moreover, let
\[
T:=a^*(t)\in H^*(F(\RR^d,p)/\ZZ/p;\FF_p),
\]
where by $a$ we have denoted the classifying map $F(\RR^d,p)/\ZZ/p\longrightarrow \BB\ZZ/p$.
It is known from~\cite[Vanishing theorem 8.2]{cohen} and~\cite[Thm.\,6.1]{blagojevic-luck-ziegler-1} that,
\[
T^{i}\neq 0  
\Longleftrightarrow
i\leq\tfrac{(d-1)(p-1)}2.
\] 
Furthermore, from of Mann and Milgram~\cite[(8.1), pp.\,264]{mann-milgram} we know the following total Chern class  
\[
c(\nu_p^{\CC})=(1+t)(1+2t)\cdots(1+(p-1)t)=1+t^{p-1}.
\]
We conclude
\[
c(\gamma_{\RR^d,p}^{\CC})=a^*(c(\nu_p^{\CC}))=1+T^{p-1}.
\]
This implies
\[
\overline{c}(\gamma_{\RR^d,p}^{\CC})=
(1+T^{p-1})^{-1}=
\sum_{m\geq 0}(-1)^m(T^{p-1})^m=
1-T^{p-1}+T^{2(p-1)}-\cdots.
\]
Since $p$ is odd and hence $M$ is even and $M \le \tfrac{(d-1)(p-1)}2$, we conclude $\overline{c}_M(\gamma_{\RR^d,p}^{\CC}) \neq 0$.
This concludes the proof of Theorem~\ref{theorem_1}.

% -------------------------------------------------------------------%
% -------------------------------------------------------------------%
\subsection{Proof of Theorem~\ref{theorem_2}}
% -------------------------------------------------------------------%
% -------------------------------------------------------------------%
The proof is presented in two steps depending on $k$.
 Let $k=\beta_1p^{r_1}+\cdots+\beta_ap^{r_a}$ where $0\leq r_1<r_2<\cdots<r_a$, and $0<\beta_i<p$ for all $1\leq i\leq a$.
Then $\alpha_p(k)=\beta_1+\cdots+\beta_a$.

\subsubsection{Step 1}
Let $k$ be a power of the prime $p$, i.e., $\alpha_p(k)=1$.
By Lemma~\ref{lemma:criterion_no_2} it suffices to prove that
\[
\overline{c}_{(d-1)(k-1)}(\xi_{\CC^d,k}^{\CC})\neq 0.
\]
Since $d=p^t$ then Theorem~\ref{theorem:hight_bound_for_odd} implies that $\h(H^*(F(\CC^d,k)/\Sym_k;\FF_p))\leq d$
for any $k$.
Consequently $c(d\,\xi_{\CC^d,k}^{\CC})=1$, and
\begin{eqnarray*}
\overline{c}(\xi_{\CC^d,k}^{\CC})&=&c((d-1)\,\xi_{\CC^d,k}^{\CC})=(1+c_1+\cdots+c_{k-1})^{d-1}\\
&=&
\sum_{\substack{i_0,\ldots,i_{k-1}\geq 0\\
       i_0+i_1+\cdots+i_{k-1}=d-1}
     }  
{d-1 \choose i_0, i_1, \ldots, i_{k-1}}\,
c_1^{i_1}c_2^{i_2}\cdots c_{k-1}^{i_{k-1}}\\
&=&
c_{k-1}^{d-1}+\sum_{\substack{i_0,\ldots,i_{k-1}\geq 0,\, i_{k-1}\leq d-2\\
       i_0+i_1+\cdots+i_{k-1}=d-1}
     }  
{d-1 \choose i_0, i_1, \ldots, i_{k-1}}\,
c_1^{i_1}c_2^{i_2}\cdots c_{k-1}^{i_{k-1}},
\end{eqnarray*}
where $c_i:=c_i(\xi_{\CC^d,k}^{\CC})$.
Here multinomial coefficients ${d-1 \choose i_0, i_1, \ldots, i_{k-1}}$ are considered modulo $p$.
Thus,
\begin{equation}
\label{eq:dual}
\overline{c}_{(d-1)(k-1)}(\xi_{\CC^d,k}^{\CC})=c_{k-1}^{d-1}\text{  and  }\overline{c}_{m}(\xi_{\CC^d,k}^{\CC})=0
\text{ for all }m\geq (d-1)(k-1)+1.
\end{equation}

\noindent
Now we use the assumption that $k$ is a power of the prime $p$ and then by Theorem~\ref{theorem : Chern classes} the monomial $c_{k-1}^{d-1}\neq 0$ does not vanish. 

\subsubsection{Step 2}
Let $k\geq1$. 
We use the morphism of fiber bundles $\prod_{t=1}^{a}\prod_{u=1}^{\beta_t}\xi_{\CC^d,p^{r_t}}$ and $\xi_{\CC^d,k}$ constructed in the proof of Theorem~\ref{theorem : Chern classes-2}. 
Recall the pullback diagram: 
 \[
\xymatrix{\prod_{t=1}^{a}\prod_{u=1}^{\beta_t}\xi_{\CC^d,p^{r_t}}\ar[r]^-{\Theta}  \ar[d] 
& \xi_{\CC^d,k}   \ar[d] 
\\
\prod_{t=1}^{a}\prod_{u=1}^{\beta_t}F(\CC^d,p^{r_t})/\Sym_{p^{r_t}}  \ar[r]_-{\theta}   & 
F(\CC^d,k)/{\Sym_k}.
}
\]
As we have seen the pullback bundle is a direct product bundle
\[
\theta^*\xi_{\CC^d,k} \cong  \prod_{t=1}^{a}\prod_{u=1}^{\beta_t}\xi_{\CC^d,p^{r_t}}.
\]

\noindent
Again, the naturality of Chern classes \cite[Lemma 14.2]{milnor-stasheff}, now applied on the inverse Chern class, combined with the Product theorem \cite[(14.7)]{milnor-stasheff}  gives  
\[
\theta^*\bar{c}(\xi_{\CC^d,k})=
\bar{c}\Big(   \prod_{t=1}^{a}\prod_{u=1}^{\beta_t}\xi_{\CC^d,p^{r_t}} \Big)=
 \bigtimes_{t=1}^{a}\bigtimes_{u=1}^{\beta_t}\bar{c}(\xi_{\CC^d,p^{r_t}})
\]
Now relation \eqref{eq:dual} from the first part of the proof implies that: 
\[
 \theta^*\bar{c}_{(d-1)(k-\alpha_p(k))}(\xi_{\CC^d,k})
 =
  \bigtimes_{t=1}^{a}\bigtimes_{u=1}^{\beta_t}\bar{c}_{(d-1)(p^{r_t}-1)}(\xi_{\CC^d,p^{r_t}}).
\]
The K\"unneth formula and relation \eqref{eq:dual} imply that
$\theta^*\bar{c}_{(d-1)(k-\alpha_p(k))}(\xi_{\CC^d,k})\neq 0$.
Consequently, 
\begin{equation}
\label{eq:dual_class}
\bar{c}_{(d-1)(k-\alpha_p(k))}(\xi_{\CC^d,k})\neq 0,
\end{equation}
and so Lemma~\ref{lemma:criterion_no_2} concludes the proof of the theorem.

%%%%%%%%%%%%%%%%%%%%%%%%%%%%%%%%%%%%%%%%%%%%%%%%%%%%%%%%%%%%%%%%%%%%%%%%%%%%%%%%%%%%%
%%%%%%%%%%%%%%%%%%%%%%%%%%%%%%%%%%%%%%%%%%%%%%%%%%%%%%%%%%%%%%%%%%%%%%%%%%%%%%%%%%%%%
\section{$\ell$-skew embeddings over $\CC$}
%%%%%%%%%%%%%%%%%%%%%%%%%%%%%%%%%%%%%%%%%%%%%%%%%%%%%%%%%%%%%%%%%%%%%%%%%%%%%%%%%%%%%
%%%%%%%%%%%%%%%%%%%%%%%%%%%%%%%%%%%%%%%%%%%%%%%%%%%%%%%%%%%%%%%%%%%%%%%%%%%%%%%%%%%%%

In this section we extend the notion of $\ell$-skew embeddings to complex manifolds. 
The $\ell$-skew embeddings over $\RR$ were previously considered by Ghomi \& Tabachnikov~\cite{ghomi-tabachnikov},  
Stojanovi\'c~\cite{stojanovic} and recently in~\cite{blagojevic-luck-ziegler-2}. 

We establish the following Chisholm-like bound for $\ell$-skew embeddings over $\CC$ under assumption that $\ell$ is an odd prime.

\begin{theorem}
  \label{theorem:Main-02}
  Let $d\geq1$ and $\ell\geq 1$ be integers.  
  There is no $\ell$-skew embedding $\CC^d\longrightarrow\CC^{N}$
  for 
  \[N< \tfrac{1}{2}\big( (2^{\gamma(d)+1}-2d-1)(\ell-\alpha(\ell))+2(d+1)\alpha(\ell)-\ell\big)-1, \]
  where  $\gamma(d)=\lfloor\log_2d\rfloor+1$.
\end{theorem}

\noindent
This result is a consequence of our bounds for real $\ell$-skew maps~\cite{blagojevic-luck-ziegler-2}.

\begin{theorem}
  \label{theorem:Main-2}
  Let $d\geq1$ be an integers and $\ell\geq 3$ be a prime.  
  There is no complex $\ell$-skew embedding $\CC^d\longrightarrow\CC^{N}$ 
  for \[N\leq (\ell-1)(d+f(d,\ell)+1)+d-1,\] where
  $f(d,\ell):=\max\{k:0\leq k\leq d-1\textrm{ and }\ell \nmid{d+k\choose d} \}$.
\end{theorem}

\noindent
For example, $f(2,3)=0$ while $f(2,\ell)=1$ for all primes $\ell\geq5$.

\begin{theorem}
  \label{theorem:Main-3}
  Let $p$ be an odd prime, $\ell\geq1$ and $d=p^t$ for $t\geq1$.  
  There is no complex $\ell$-skew embedding $\CC^d\longrightarrow\CC^{N}$
  for \[N\leq (d-1)(\ell-\alpha_p(\ell))+(d+1)\ell-2.\]
\end{theorem}

%-----------------------------------------------------------------------------------%
\subsection{Definition and the first bound}
%-----------------------------------------------------------------------------------%

The complex affine subspaces $L_1,\ldots,L_{\ell}$ of the complex vector space $\CC^N$ are
\emph{affinely independent} if the affine span of their union has affine complex dimension  $(\dimaff L_1+1)+\cdots+(\dimaff L_{\ell}+1)-1$.  

For a complex $d$-dimensional manifold $M$ we denote by $TM$ the complex tangent bundle 
of $M$ and by $T_yM$ the corresponding tangent space to $M$ at the point $y\in M$.

\begin{definition}[Skew embedding over $\CC$]\label{def:skew_map}
  Let $\ell\geq 1$ be an integer, and $M$ be a smooth complex $d$-dimensional manifold.  
  A smooth complex embedding $f \colon M\longrightarrow\CC^N$ is an \emph{$\ell$-skew embedding over $\CC$} if for every $(y_1,\ldots,y_{\ell})\in F(M,\ell)$
  the complex affine subspaces
  \[
  (\iota\circ df_{y_1})(T_{y_1}M),\ldots ,(\iota\circ df_{y_{\ell}})(T_{y_{\ell}}M)
  \]
  of $\CC^N$ are affinely independent.
\end{definition}

\noindent
Here $df\colon TM\longrightarrow T\CC^N$ denotes the complex differential map between tangent complex vector bundles induced by $f$, and 
\begin{equation}
\label{iota}
\iota \colon T\CC^N  \longrightarrow \CC^N
\end{equation}
sends a tangent vector $v \in T_x\CC^N$ at $x \in \CC^N$
to  $x +v$ where we use the standard identification $T_x\CC^N = \CC^N$.

The first bound for the existence of an complex $\ell$-skew embedding $f \colon M\longrightarrow\CC^N$  we get by dimension count as in~\cite[Lemma 3.3]{blagojevic-luck-ziegler-2}.
\begin{lemma}
Let $\ell\geq 1$ be an integer, $M$ be a smooth complex $d$-dimensional manifold, and $f \colon M\longrightarrow\RR^N$ be a complex $\ell$-skew embedding.
Then $(d+1)\ell-1\leq N$.
\end{lemma}

%-----------------------------------------------------------------------------------%
\subsection{A topological criterion}
%-----------------------------------------------------------------------------------%

Following~\cite[Section 3.2]{blagojevic-luck-ziegler-2} we derive a necessary condition for the existence of a complex $\ell$-skew embedding.

Let $M$ be a smooth complex $d$-dimensional manifold.  The symmetric group $\Sym_{\ell}$
acts naturally on the configuration space $F(M,\ell)$ and its tangent bundle $TF(M,\ell)$
which is a complex vector bundle.  
The argument for the real case and Stiefel-Whitney classes appearing in the proof 
of~\cite[Lemma~3.4]{blagojevic-luck-ziegler-2} carries over to the complex case and 
Chern classes modulo $p$ and hence we get for the vector bundle
$\xi_{M,\ell}^{\CC}  \colon F(X,\ell)\times_{\Sym_\ell}\CC^{\ell} \longrightarrow F(X,\ell)/\Sym_\ell$.

\begin{lemma}
  \label{lemma:TK-l-2}
  Let $d,\ell\geq 1$ be integers and let $M$ be a smooth complex $d$-dimensional manifold. 
  \begin{compactenum}[\rm (i)]
  \item If there exists a complex $\ell$-skew embedding  $M\longrightarrow\CC^N$, then the complex vector bundle
            $T(F(M,\ell)/\Sym_{\ell})\oplus\xi_{M,\ell}^{\CC}$
            over the unordered configuration space $F(M,\ell)/\Sym_{\ell}$ admits an $(N-(d+1)\ell+1)$-dimensional complex inverse.
  \item If the inverse Chern class 
    \[
     \overline{c}_{N - (d+1)\ell +2}\bigl(T(F(M,\ell)/\Sym_{\ell})\oplus\xi_{M,\ell}^{\CC}\bigr)
     \] 
     does not  vanish, then there is no complex $\ell$-skew embedding $M\longrightarrow\CC^{N}$.
  \end{compactenum}
  
\end{lemma}

%-----------------------------------------------------------------------------------%
\subsection{Proof of theorem~\ref{theorem:Main-02}}
%-----------------------------------------------------------------------------------%
Let  $\CC^d\longrightarrow\CC^N$ be a complex $\ell$-skew embedding.
Then, by Lemma~\ref{lemma:TK-l-2} the complex vector bundle 
\[
T(F(\CC^d,\ell)/\Sym_{\ell})\oplus\xi_{\CC^d,\ell}^{\CC}\cong (d+1)\xi_{\CC^d,\ell}^{\CC}
\] 
admits an $(N-(d+1)\ell+1)$-dimensional complex inverse. 
Thus, the real vector bundle $2(d+1)\xi_{\RR^{2d},\ell}^{\RR}$ admits a $2(N-(d+1)\ell+1)$-dimensional real inverse.
Consequently, $(2d+1)\xi_{\RR^{2d},\ell}^{\RR}$ admits a $(2N-2(d+1)\ell+2+\ell)$-dimensional real inverse.
Now,  by \cite[Thm.\,3.7]{blagojevic-luck-ziegler-2}, $\bar{w}_{(2^{\gamma(2d)}-2d-1)(\ell-\alpha(\ell))}(\xi_{\RR^{2d},\ell}^{\RR})\neq 0$ and so:
\begin{multline*}
2N-2(d+1)\ell+2+\ell\geq (2^{\gamma(2d)}-2d-1)(\ell-\alpha(\ell))\Longleftrightarrow \\
N\geq \tfrac{1}{2}\big( (2^{\gamma(d)+1}-2d-1)(\ell-\alpha(\ell))+2(d+1)\alpha(\ell)-\ell\big)-1.
\end{multline*}

%-----------------------------------------------------------------------------------%
\subsection{Proof of theorem~\ref{theorem:Main-2}}
%-----------------------------------------------------------------------------------%
To prove the theorem, according to Lemma~\ref{lemma:TK-l-2}, it suffices to show that the inverse Chern class
\[
\overline{c}_{(\ell-1)f(d,\ell)}\bigl(T(F(\CC^d,\ell)/\Sym_{\ell})\oplus\xi_{\CC^d,\ell}^{\CC}\bigr).
\]
does not vanish. 
Since $T(F(\CC^d,\ell)/\Sym_{\ell})\approx d\,\xi_{\CC^d,\ell}^{\CC}$ we need to prove that
\[
\overline{c}_{(\ell-1)f(d,\ell)}((d+1)\,\xi_{\CC^d,\ell}^{\CC})\neq 0.
\]
Indeed, recall from the proof of Theorem~\ref{theorem_1} that
\[
\overline{c}(\gamma_{\CC^d,\ell}^{\CC})=
r^*\overline{c}(\zeta_{\CC^d,\ell}^{\CC})=
r^*\overline{c}(\xi_{\CC^d,\ell}^{\CC})
\,\textrm{ and }\,
\overline{c}(\gamma_{\CC^d,\ell}^{\CC})=
(1+T^{\ell-1})^{-1}.
\]
Then 
\[
\overline{c}((d+1)\,\gamma_{\CC^d,\ell}^{\CC})=
(1+T^{\ell-1})^{-(d+1)}=
\sum_{k=0}^{d-1}{d+k\choose d}T^{(\ell-1)k}.
\]
Consequently,  $\overline{c}_{(\ell-1)f(d,\ell)}((d+1)\,\gamma_{\CC^d,\ell}^{\CC})\neq 0$
and so  $\overline{c}_{(\ell-1)f(d,\ell)}((d+1)\,\xi_{\CC^d,\ell}^{\CC})\neq 0$.

%-----------------------------------------------------------------------------------%
\subsection{Proof of theorem~\ref{theorem:Main-3}}
%-----------------------------------------------------------------------------------%
Like in the proof of Theorem~\ref{theorem:Main-2}, from Lemma~\ref{lemma:TK-l-2} and observation
$T(F(\CC^d,\ell)/\Sym_{\ell})\approx d\,\xi_{\CC^d,\ell}^{\CC}$,  the relevant  inverse Chern class to study is $\overline{c}\big((d+1)\,\xi_{\CC^d,\ell}^{\CC}\big)$.

\noindent
Since by an assumption $d=p^t,$ as in the proof of Theorem~\ref{theorem_2}, we have that:
\[
c(d\,\xi_{\CC^d,\ell}^{\CC})=1\quad \Longrightarrow \quad\overline{c}(d\,\xi_{\CC^d,\ell}^{\CC})=1 
\quad\Longrightarrow \quad\overline{c}((d+1)\,\xi_{\CC^d,\ell}^{\CC})=\overline{c}(\xi_{\CC^d,\ell}^{\CC}).
\]
The relation \eqref{eq:dual_class} implies that
\[
\overline{c}_{(d-1)(\ell-\alpha_p(\ell))}((d+1)\,\xi_{\CC^d,\ell}^{\CC})=\overline{c}_{(d-1)(\ell-\alpha_p(\ell))}(\xi_{\CC^d,\ell}^{\CC})\neq 0.
\]
Now from Lemma~\ref{lemma:TK-l-2} we conclude that there can not be any $\ell$-skew embedding $\CC^d\longrightarrow\CC^N$ over $\CC$ for $N\leq (d-1)(\ell-\alpha_p(\ell))+(d+1)\ell-2$.

%-----------------------------------------------------------------------------------%

\end{document}